\documentclass[12pt,reqno]{amsart}
\usepackage[latin1]{inputenc}
\usepackage{amscd}
\usepackage{amsmath,amsthm,amssymb,amsfonts}
\usepackage{mathrsfs}
\usepackage{graphicx}
\usepackage{fancyhdr}
\usepackage{stmaryrd}

\usepackage{color}
\usepackage{amsmath}
\usepackage{amsfonts}
\usepackage{amssymb}
\usepackage{geometry}
\newtheorem{theorem}{Theorem}[section]
\newtheorem{proposition}{Proposition}[section]
\newtheorem{lemma}{Lemma}[section]

\theoremstyle{Definition}
\newtheorem{definition}{Definition}[section]

\def\R{\mathbb{R}}

\def\E{\mathcal E}

\def \p{\partial}

\def \e {\varepsilon}
\def \I {\mathcal I}

\def \O {\Omega}
\DeclareMathOperator{\cof}{cof}

\DeclareMathOperator{\dist}{dist}

\DeclareMathOperator{\dive}{div}
\DeclareMathOperator{\Det}{Det}
\DeclareMathOperator{\supp}{supp}
\DeclareMathOperator{\tr}{tr}
\DeclareMathOperator{\Dive}{Div}
\title[On the existence of minimizers for the Neo-hookean energy]{On the Existence of minimizers for the Neo-hookean energy in the axisymmetric setting}
\author{Duvan Henao and R\'emy Rodiac}
\date{}
\address[D.Henao]{Facultad de Matematic\'as, Pontificia Universidad Cat\'olica de Chile, Vicu\~na Mackenna 4860, Macul, Santiago Chile}
\email{dhenao@mat.puc.cl}
\address[R.Rodiac]{Facultad de Matematic\'as, Pontificia Universidad Cat\'olica de Chile, Vicu\~na Mackenna 4860, Macul, Santiago Chile}
\email{remy.rodiac@mat.uc.cl}

\begin{document}

\begin{abstract}
Let $\O$ be a smooth bounded axisymmetric set in $\R^3$. In this paper we investigate the existence of minimizers of the so-called neo-Hookean energy among a class of axisymmetric maps. Due to the appearance of a critical exponent in the energy we must face a problem of lack of compactness. Indeed as shown by an example of Conti-De Lellis in \cite[Section 6]{ContiDeLellis}, a phenomenon of concentration of energy can occur preventing the strong convergence in $W^{1,2}(\O,\R^3)$ of a minimizing sequence along with the equi-integrability of the cofactors of that sequence. We prove that this phenomenon can only take place on the axis of symmetry of the domain. Thus if we consider domains that do not contain the axis of symmetry then minimizers do exist. We also provide a partial description of the lack of compactness in terms of Cartesian currents. Then we study the case where $\O$ is not necessarily axisymmetric but the boundary data is affine. In that case if we do not allow cavitation (nor in the interior neither at the boundary) then the affine extension is the unique minimizer, that is, quadratic polyconvex energies are $W^{1,2}$-quasiconvex in our admissible space. At last, in the case of an axisymmetric domain not containing its symmetry axis, we obtain for the first time the existence of weak solutions of the energy-momentum equations for $3$D neo-Hookean materials.\\


\end{abstract}

\maketitle

\section{Introduction}

The most commonly used model to describe the nonlinear behavior of elastic solids undergoing large deformations is that of neo-Hookean materials, whose stored energy is assumed to be of the form (see e.g.\ \cite{treloar1975physics})

\begin{equation}\label{E}
 E(u)= \int_\O | D u|^2 + H( \det D u),
\end{equation}
where $\O\subset \R^3$ is the reference configuration, $u:\O \rightarrow \R^3$ is the deformation map and $H : \R \rightarrow \R^+$ is a (smooth) convex function satisfying
\begin{equation}\label{propH}
 \lim_{t\rightarrow +\infty} \frac{H(t)}{t}=\lim_{s\rightarrow 0} H(s)= +\infty.
\end{equation}

In spite of this model being broadly used by mathematicians, physicists, materials scientists and engineers alike, and in spite of it being justifiable from statistical mechanics (cf.\ \cite{treloar1975physics}), the existence of stable configurations (i.e.\ minimizers of $E$ in an appropriate function space) remains a mathematical challenge. If the quadratic exponent in $E(u)$ is replaced by any exponent $p>2$, the existence of minimizers (in the Sobolev setting) has already been established (cf.\ \cite{Ball77}, \cite{MullerTangYang94}, \cite{MullerSpector1995}, \cite{SivaloganathanSpector}). However in the borderline case $p=2$ there are delicate issues of lack of compactness, as reported by Conti-De Lellis \cite{ContiDeLellis}, which we explain below.

We recall (cf.\ \cite{Mullercours} or \cite{Struwe}) that there are two important steps in the method of calculus of variations in order to minimize a functional $I$ in a space $X$:

\begin{itemize}
\item[i)] Compactness of minimizing sequences: let $(u_n)_n$ be a minimizing sequence for $I$ in $X$, then there exists $u \in X$ such that $u_n \overset{\tau}{\rightarrow} u$ in $X$, where $\overset{\tau}{\rightarrow}$ refers to the convergence for a suitable topology in X.

\item[ii)] Lower semi-continuity: for a minimizing sequence $(u_n)_n$ such that $u_n \overset{\tau}{\rightarrow} u$ in $X$ we have that
    \[ I(u) \leq \lim_{n \rightarrow +\infty} I(u_n). \]
\end{itemize}

We look for minimizers of $E$ in a class of deformations which satisfy some invertibility conditions (in order to avoid interpenetration of matter), which are orientation preserving and have prescribed boundary data (pure displacement problem).
We let $M>0$, $\O,\O' \subset \R^3$ be smooth bounded domains, $g: \overline{\O}\rightarrow \overline{\O'}$ be a $C^1$-diffeomorphism and

 \begin{eqnarray}
 \mathcal{A} &:= & \{ u\in W^{1,2}\cap L^\infty(\O,\R^3);\ u_{|\partial \O}=g_{|\partial \O}, \ \|u\|_\infty \leq M,   \nonumber \\
 & &  \phantom{aaaaaaaaaaaaaa} \det D u>0 \text{ a.e.\ and u is one-to-one a.e.}\}. \nonumber
\end{eqnarray}

A first difficulty is that we can not reasonably hope for lower semicontinuity of $E$ in the space $\mathcal{A}$. Indeed the space $\mathcal{A}$ allows the possibility of \emph{cavitation}, i.e.\ formation of holes in the deformed configuration $\O'$ (see \cite{BallDiscontinuous} for a presentation of the phenomenon of cavitation in the radial setting). In \cite{BallMurat}, Ball-Murat proved that if cavitation is energetically favorable then the functional \eqref{E} is not weakly lower semicontinuous. More precisely if there is $\lambda >0$ and a map $u\in W^{1,2}$ such that, for the unit ball $B$
 \[ \int_B | D u|^2 + H( \det D u) < \int_B |D(\lambda \text{Id})|^2+ H( \det D (\lambda \text{Id})), \ \ \text{with } u=\lambda\text{Id} \ \text{on } \partial B,\] then by suitably rescaling $u$ and covering $B$ with small balls one can construct a sequence such that
\[ u_k \rightharpoonup \lambda \text{Id} \ \ \text{ in } W^{1,2}, \]
\[ \lim_{k\rightarrow +\infty} \int_B | D u_k|^2 + H( \det D u_k)= \int_B | D u|^2 + H( \det D u)< E(\lambda\text{Id}). \]
Such a bad sequence corresponds to the creation of many very small cavities of (almost) constant total volume.

 As mentioned before, working with energies of the form
 \begin{equation}\label{E_p}
  E_p(u)= \int_\O | D u|^p + H( \det D u), \ \ \ \text{with } p>2
 \end{equation}
 Müller-Spector proposed to consider a stored energy which is the sum of \eqref{E_p} and of $\text{Per } u(\Omega)$, here $\text{Per}$ denotes the perimeter of a set and $u(\O)$ has to be defined in a certain precise sense (we need to consider the geometrical image of $u$ cf.\ Definition 2.7 in \cite{ContiDeLellis}). The last term penalizes the creation of surface and hence cavitation. They also introduced a notion of invertibility called condition (INV) which is stable under weak convergence in $W^{1,p}$ for $p>2$ and which allows them to recover lower semicontinuity of energies of the type $E_p$ and compactness of minimizing sequences in the appropriate space. The formulation of that condition relies on the topological degree (cf.\ \cite{degreeandBMO} for more on the degree). In \cite{ContiDeLellis}, using the definition of the degree in $W^{1,2}\cap L^\infty(\Omega,\R^3)$, Conti-De Lellis extended condition (INV) to deformations belonging to that space. We refer to Definition 3.6 in \cite{ContiDeLellis} and Definition 3.2 in \cite{MullerSpector1995} for the precise definition of condition (INV). However they observed that this condition is no longer closed under weak $W^{1,2}$ convergence by giving an explicit example (Section 6 in \cite{ContiDeLellis}). Thus minimizing the energy $E$ in a space of deformations which satisfy condition (INV) is a problem with lack of compactness.

Another approach developed by Henao--Mora-Corral consists in refining the notion of the surface energy that measures the surface created by $u$ in the deformed configuration. For $f\in C^\infty_c(\O\times \R^3,\R^3)$, if $u\in W^{1,2}$ is such that $\det D u \in L^1(\O)$, we define
\[ \mathcal{E}_u(f)=\int_\O \langle\cof D u(x), D_xf(x,u(x))\rangle+\det D u(x)\dive_yf(x,u(x)) dx\]
and
\[ \mathcal{E}(u)= \sup \{\mathcal{E}_u(f), \ f \in C^\infty_c(\O\times\R^3,\R^3), \ \|f\|_\infty \leq 1\}.\]
Then they consider the minimization of a stored energy of the form
\[G(u)= \int_\O W(x,u(x),Du(x)) +\lambda \mathcal{E}(u),  \]
in the space
\[\mathcal{A}_p:=\{ u \in W^{1,p}(\O,\R^3); \det Du >0 \text{ a.e }, u \text{ is one-to-one a.e.}, u_{|\partial \Omega}=g_{|\partial \Omega} \}\]
with $\lambda>0$ and  \[W(x,y,F) \geq a(x)+c|F|^p+h_1(|\text{cof} F|)+h_2(\det F) \]
for a.e.\ $x\in \O$, all $y\in \R^3$ and all $F \in R^{3\times 3}$, for some $p\geq 2$, $a\in L^1(\O)$, some constant $c>0$, an increasing function $h_1:(0,\infty) \rightarrow [0,\infty)$ and a convex function $h_2:(0,\infty) \rightarrow \R$ such that
\[ \lim_{t\rightarrow +\infty} \frac{h_1(t)}{t}=\lim_{t\rightarrow +\infty} \frac{h_2(t)}{t}=\lim_{t\rightarrow 0^+}h_2(t)=+\infty.\]
The introduction of the surface energy $\mathcal{E}$ allows them to recover lower semicontinuity and compactness of the problem without using a supplementary condition like (INV). Indeed they proved that under certain conditions including a uniform bound on the surface energy $\mathcal{E}$, a sequence of one-to-one almost everywhere deformations converging pointwise is such that the limit is also one-to-one almost everywhere (see Theorem 2 in \cite{HenaoMoraCorral2010}). However their result does not include the minimization of $E$. The same example of Conti-De Lellis shows that the minimization of $E$ in the space
\[ \mathcal{A}_{nc} = \{ u\in \mathcal{A}, \ \mathcal{E}(u)=0\}, \]
for example, is also a problem with lack of compactness. Indeed, for a deformation $u\in W^{1,2}\cap L^\infty(\O,\R^3)$, $\mathcal{E}(u)=0$ implies that $\Det Du=\det Du$ (see Proposition 7.1 in \cite{HenaoMoraCorral2012}). Here $\Det Du$ is the distributional determinant ($\Det Du:= \frac 13 \dive[ (\cof Du)^T.u ]$), whereas $\det Du$ is the determinant defined pointwise. Now Conti-De Lellis constructed a sequence $(u_n)$ of bi-Lipschitz homeomorphisms (hence satisfying $\mathcal{E}(u)=0$ cf.\ Lemma 5.3 in \cite{HenaoMoraCorral2012}) such that $u_n \rightharpoonup u$ in $W^{1,2}$ and $\Det Du=\det Du +\frac{\pi}{6}\delta_p-\frac{\pi}{6}\delta_O$ for two points $P,O$ in the domain they considered, hence $\mathcal{E}(u)\neq 0$. This example corresponds to what is called a \emph{dipole} phenomenon, in harmonic maps theory for example. The notion of dipole was introduced in \cite{BrezisCoronLieb}. It is a concentration phenomenon of the Dirichlet energy on a segment or a curve. \\

Thus we can address the problem of finding stable configurations for neo-Hookean materials through (at least) two different approaches. Both approaches leading to two problems, one without cavitation and one with cavitation. If we use the approach of Henao--Mora-Corral these two problems are: \\

\textbf{Problem 1}: Minimize the energy $E(u)$ for $u$ in the space $\mathcal{A}_{nc}$. \\

\textbf{Problem 2}: Minimize the energy $E(u)+\lambda \mathcal E(u)$ in the space $\mathcal{A}$, for some $\lambda>0$.\\

In order to translate these problems in the approach of M\"uller-Spector and Conti-De Lellis we introduce

\begin{eqnarray}
\I &:=& \{u\in W^{1,2}\cap L^\infty(\O,\R^3); u_{|\p \O}=g_{|\p \O}, \|u\|_\infty \leq M, \nonumber \\
  & &  \phantom{aaaaaaaaaaaaaa} \det Du> 0 \ \text{ a.e.\ and } u \text{ satisfies  (INV)} \} \nonumber
\end{eqnarray}
and
\begin{equation}\nonumber
\I_{nc}:= \{ u \in \I; \  \det Du=\Det Du \}.
\end{equation}
Then we are led to: \\

\textbf{Problem 1'}: Minimize the energy $E(u)$ for $u$ in the space $\mathcal{I}_{nc}$. \\

\textbf{Problem 2'}: Minimize the energy $E(u)+\lambda \text{Per}(im_G(u))$ in the space $\mathcal{I}$, for some $\lambda>0$.\\

 The difficulties in both approaches are similar. Indeed, in light of Theorem 1 of \cite{HenaoMoraCorral2010} and Theorem 8.5 of \cite{HenaoMoraCorral2012}, a crucial ingredient for the lower semicontinuity of $\mathcal{E}$ or for the condition (INV) to pass to the weak limit in $W^{1,2}$ is the equi-integrability of the cofactors. We can check that the cofactors are not equi-integrable in the example of Conti-De Lellis, they concentrate on a segment. For conciseness, we shall focus only on the first approach. We note that the example of Conti-De Lellis is axisymmetric. Thus the axisymmetric setting seems to contain already the difficulties of the problem although it is simpler. That is why we choose to study the existence of minimizers of the energy $E$ in the space of deformations that belong to $\mathcal{A}_{nc}$ and are also axisymmetric.

\begin{definition}
Let $\O \subset \R^3$ be a bounded axisymmetric domain (with respect to the $z$ axis). Let $\O_0:=\O \cap xOz$ ($xOz$ denotes the plane passing through the origin and containing the vectors $e_x$ and $e_z$). We say that $u:\O \rightarrow \R^3$ is axisymmetric if there exists $v=(v_1,v_2): \O_0 \rightarrow \R^2$ such that
\begin{equation}\nonumber
u(r\cos\theta, r\sin \theta,z)=v_1(r,z)(\cos \theta e_x+\sin \theta e_y)+v_2(r,z)e_z, \ \ \text{with } v_1\geq 0 \text{ a.e}.
\end{equation}
\end{definition}

\begin{definition}
Let $\O \subset \R^3$ and $\O' \subset \R^3$ be smooth bounded axisymmetric domains. Let $g: \overline{\O}\rightarrow \overline{\O'}$ be a $C^1$ axisymmetric diffeomorphism. We set
\begin{equation}\nonumber
\mathcal{A}^{axi}:=\{ w \in \mathcal{A}; w \text{ is axisymmetric} \}
\end{equation}
and
\begin{equation}\nonumber
\mathcal{A}_{nc}^{axi}:=\{ w \in \mathcal{A}_{nc}; w \text{ is axisymmetric} \}.
\end{equation}
\end{definition}

The fundamental question in this article is the following:
\[ \text{ Is } \inf\{E(w); w \in \mathcal{A}_{nc}^{axi}\} \text{ attained ? }\]

We now state our main result. We denote by $O_z$ the $z$ axis.

\begin{theorem}\label{main}
Let $(u_n)_n$ be a minimizing sequence for $E$ in $\mathcal{A}_{nc}^{axi}$. Then
\begin{equation}\nonumber
\cof Du_n \rightharpoonup \cof Du \ \text{ in } L^1(K) \text{ for every } K\subset (\O\setminus\{O_z\}) \text{ compact }.
\end{equation}
Furthermore if $\O$ is an axisymmetric domain such that $\inf_{(x,y,z)\in \O} \sqrt{x^2+y^2}>0$ then there exists a minimizer of $E$ in $\mathcal{A}_{nc}^{axi}$.
\end{theorem}

The spirit of the proof of this theorem is the following. Working in an axisymmetric setting in $\R^3$ reduces the problem to an ``almost" two-dimensional problem or a two dimensional problem ``away from the axis". Indeed expressing the cofactor matrix in cylindrical coordinates, we see that the problem reduces to the convergence of a $2\times 2$ determinant. We can then use a theorem of Müller \cite{Higherintegrability} about the weak $L^1$ convergence of positive determinant of transformations $v\in W^{1,n}(U,\R^n)$ with $U\subset \R^n$ to obtain the weak $L^1$ convergence of the cofactors away from the axis.
The weak $L^1$ convergence of the $3\times3$ determinant follows then from Theorem 3 in \cite{HenaoMoraCorral2010}. We provide an alternative proof of this fact using a result of De Lellis-Ghiraldin, which states that under some conditions, if $\Det D u$ is a Radon measure then its absolutely continuous part is equal to $\det D u$ almost everywhere.


Once we obtain the existence of minimizers in the class $\mathcal{A}_{nc}^{axi}$ a natural question is to know if these minimizers satisfy some equations. In elasticity it is a long standing open problem to know if minimizers do satisfy the Euler-Lagrange equations associated to the stored energy. However we expect these minimizers to satisfy other equations like inner-variational equations (see, e.g.\ \cite{Ballmin}). Our second result states that the minimizers we obtained do satisfy the inner-variational equations associated with the neo-Hookean energy $E$. Thus we provide the first existence results of nontrivial solutions to these equations.

\begin{theorem}
Let $H$ be as before a $C^1$ convex function satisfying \eqref{propH}. Assume furthermore that there exist $s>0, c_1,c_2>0, d_0>0$ such that
\begin{equation}\nonumber
c_1t^{-s-k}\leq (-1)^k\frac{d^k}{dt^k}H(t)\leq c_2t^{-s-k} \ \ \ \ \text{ for } k=0,1 \ \ \text{ and for } t<d_0,
\end{equation}
and there exist $\tau>0$, $c_3,c_4>0$ and $d_1>0$ such that
\begin{equation}\nonumber
c_3t^{\tau+1} \leq H'(t) \leq c_4t^{\tau+1} \ \ \ \ \text{ for } t\geq d_1.
\end{equation}
Let $u$ be a minimizer of $E$ in $\mathcal{A}_{nc}^{axi}$. Then $u$ satisfies
\begin{equation}
\Dive\left(2Du^TDu+\left[H'(\det Du)\det Du-|Du|^2-H(\det Du)\right]I\right)=0. \nonumber
\end{equation}
That last equation can also be written
\begin{eqnarray}
\sum_{i=1}^3 \p_i\left[ 2\p_iu\cdot \p_ju -\left( H'(\det Du)\det Du-|Du|^2-H(\det Du) \right)\delta_{ij}\right]=0, \nonumber \\
\phantom{aaaaaaaaaaaaaaaaaaaaaa} \text{ for } j=1,...,3. \nonumber
\end{eqnarray}
\end{theorem}

The paper is organized as follows: in Section 2 we recall the notations and definitions needed in the sequel. In Section 3 we study the weak convergence of the cofactors and determinants associated to a minimizing sequence for $E$ in $\mathcal{A}_{nc}^{axi}$. This rests upon  computations of the energy $E$, the cofactors and the determinants in cylindrical coordinates. For the comfort of the reader we provide these computations in the Appendix. This weak convergence is used to obtain minimizers of the problem. In Section 4 we discuss the lack of compactness of the problem from the point of view of geometric measure theory and currents. In Section 5, we study the case where the boundary data is affine, but for a general domain, and prove that, despite the lack of compactness of the problem, minimizers do exist in this situation if we work in a class where no cavitation can occur (in the interior and at the boundary of the domain). At last we prove that the minimizers obtained satisfy the inner-variational equations of the neo-Hookean energy $E$.

\section{Notations and Definitions}
Let $n \geq2$, the space of $n\times n$ matrices with real coefficients is denoted by $M_n(\R)$. The identity matrix is denoted by $I$. \\
For two vectors $a,b$ in $\R^n$,  $a\cdot b$ denotes their inner product. The inner product of two matrices $A,B$ in $M_n(\R)$ is defined by $\langle A,B\rangle =\tr (A^TB)$, where $\tr$ denotes the trace of a matrix and $A^T$ is the transpose of $A$. \\
The cofactor matrix of $A$ is denoted by $\cof A$ and satisfies $A^T\cof A=(\det A)I$. For $E,F$ two Banach spaces and a $C^1$ function $f:E\rightarrow F$, we denote by $Df(x)$ the differential of $f$ at $x$ and by $Df(x).h$ the differential of $f$ at $x$ applied to the vector $h\in E$. \\
 We use the notations $x,y,z$ for cartesian coordinates in $\R^3$, and $(r,\theta,z)$ for cylindrical coordinates ($x=r\cos \theta,y=r\sin\theta$). We let $X=(r,z)$ in $ \R^+\times \R$. We use $(e_x,e_y,e_z)$ as a notation for the canonical basis of $\R^3$ and we let $e_r:=\cos\theta e_x+\sin \theta e_y$ and $e_\theta:=-\sin \theta e_x+ \cos \theta e_y$. \\
 We denote by $a\wedge b=a_1b_2-a_2b_1$ the determinant of two vectors in $\R^2$. \\
 For general vectors $a,b$ in $\R^n$ the tensor product $a \otimes b$ is a tensor (which can be viewed as a matrix) defined by its action $a\otimes b.h= (b\cdot h)a$ for $h$ in $\R^n$. \\
 The divergence operator is denoted by $\Dive$ in the reference configuration, and by $\dive$ in the deformed configuration. More precisely, the expression $\Dive \phi$ is used for functions $\phi$ defined on $\O$ (i.e.\ on the $x$ variables), while $\dive g$ is used for functions $g$ defined on the target space (so the differentiation is with respect to $y$).
\section{Proof of Theorem \ref{main}}

Let $\O\subset \R^3$ be a smooth axisymmetric bounded domain and $u:\O \rightarrow \R^3$ an axisymmetric deformation. We can write
\[u(r\cos \theta, r\sin \theta,z)=v_1(r,z)(\cos\theta e_x+\sin \theta e_y)+ v_2(r,z)e_z,\]
for some $v=(v_1,v_2):\O_0 \rightarrow \R^2$ with $v_1 \geq 0$ almost everywhere.  We can express the Jacobian matrix $Du$ in the basis $(e_r,e_\theta,e_z)$ as (cf.\ Appendix):
\begin{equation}\nonumber
Du=\begin{pmatrix}
\p_rv^1 & 0 & \p_z v^1 \\
0 & \frac{v^1}{r} & 0 \\
\p_rv^2 & 0 & \p_z v^2
\end{pmatrix}.
\end{equation}
We thus obtain that the energy of $E$ can be written as
\begin{equation}\nonumber
E(u)=2\pi G(v)
\end{equation}
with
\begin{equation}\label{G}
G(v) =\int_{\O_0} (|\p_rv|^2+|\p_zv|^2)rdrdz +\int_{\O_0} \frac{v_1^2}{r}drdz+ \int_{\O_0}H(\frac{v_1}{r}\det D v)rdrdz.
\end{equation}

We also have that
\begin{equation}\nonumber
\det Du= \frac{v_1}{r}\det D v,
\end{equation}
with $D v= \begin{pmatrix}
\p_r v_1 & \p_z v_1 \\
\p_r v_2 & \p_z v_2
\end{pmatrix}$
\begin{equation}\nonumber
\cof Du=\begin{pmatrix}
\frac{1}{r}v_1\p_zv_2 & 0 & -\frac{1}{r}v_1\p_zv_1 \\
0 & \det D v & 0 \\
-\frac{1}{r}v_1\p_rv_2 & 0 & \frac{1}{r}v_1\p_rv_1
\end{pmatrix}
\end{equation}
where we also expressed the cofactor matrix in the basis $(e_r,e_\theta,e_z)$. Another quantity which will play a role is the vector field $\mathcal{D}(u):= (\cof Du)^T.u$ which can also be written in the basis $(e_r,e_\theta,e_z)$ as
\begin{equation}\nonumber
\mathcal{D}(u)=\frac{v_1}{r}(v\wedge \p_zv,0,-v\wedge \p_rv).
\end{equation}

We now proceed to the proof of Theorem \ref{main} which we divide in several propositions describing the behavior of minimizing sequences. Note that since we assume that the trace of the deformations in $\mathcal{A}_{nc}^{axi}$ is the restriction to $\p \Omega$ of a $C^1$ diffeomorphism we have that $$0< \inf\{E(w); w \in \mathcal{A}_{nc}^{axi} \}<+\infty.$$ We begin with

\begin{proposition}\label{cofconv}
Let $(u_n)_n$ be a minimizing sequence for $E$ in $\mathcal{A}_{nc}^{axi}$ then
\begin{equation}\nonumber
\cof Du_n \rightharpoonup \cof Du \ \ \ \text{ in } L^1(K) \ \text{ for every K compact such that } K\subset \O \setminus \{O_z\}.
\end{equation}
\end{proposition}

\begin{proof}
Let $u_n=v_1^ne_r+v_2^ne_\theta$ be a minimizing sequence in $\mathcal{A}_{nc}^{axi}$. We have that $G(v_n)$ is bounded (where $G$ is defined by \eqref{G}), and hence we deduce that
\begin{itemize}
\item[*] $|D v_n|\sqrt{r}$ is bounded in $L^2(\O_0)$
\item[*] $ \frac{v_1^n}{\sqrt{r}}$ is bounded in $L^2(\O_0)$.
\end{itemize}
In particular we have that $v_n \rightharpoonup v$ for some $v$ in $H^1_{\text{loc}}(\O_0 \setminus \{O_z\})$. Indeed let  $K$ be an arbitrary compact set contained in $\O_0 \setminus \{O_z\}$, there exist $0<r_0<R_0$ such that $r_0\leq r\leq R_0$ in $K$. Thus $|D v_n|$ is bounded in $L^2(K)$ and $v_n$ is bounded in $L^2(K)$ (recall that $v_n$ is bounded in $L^\infty(\O_0))$. Now because we have \[\det Du_n=\frac{1}{r}v_1^n\det D v_n>0 \text{ a.e. } \]  and we assumed that $v_1^n\geq 0$ a.e.\ (in the definition of axisymmetry) we deduce that
\[\det D v_n >0 \text{ a.e}. \]
We can thus apply a result of Müller \cite{Higherintegrability} (see also \cite{Hardyspaces}) to obtain that
\begin{equation}\nonumber
\det D v_n \rightharpoonup \det D v \text{ in } L^1(K).
\end{equation}
To check that $\cof Du_n \rightharpoonup \cof Du$ in $L^1(K)$  it remains to prove the same for the other entries of the matrix $\cof Du_n$. We only treat the term $\frac{1}{r}v_1^n\p_zv_2^n$ the proof being identical for the other terms. \\
Since $v_n$ is bounded in $L^\infty(\O_0)$ and $|D v_n|$ is bounded in $L^2(K)$  we have that $v_1^n\p_zv_2^n$ is bounded in $L^2(K)$ and hence converges weakly in that space. But thanks to the Sobolev injections we have that $v_n \rightarrow v$ strongly in $L^2(K)$. Since $D v_n \rightharpoonup D v$ in $L^2(K)$ we find that
\[ \frac{1}{r}v_1^n\p_zv_2^n \rightharpoonup \frac{1}{r}v_1\p_zv_2 \ \text{ in } L^2(K).\]
This concludes the proof.

\end{proof}

\begin{proposition}\label{detconv}
Let $(u_n)_n$ be a minimizing sequence for $E$ in $\mathcal{A}_{nc}^{axi}$ we have that
\begin{equation}\nonumber
\det Du_n \rightharpoonup \det Du \text{ in } L^1(\O)
\end{equation}
\begin{equation}\nonumber
u \text{ is one-to-one almost everywhere in } \O.
\end{equation}
\end{proposition}

\textbf{Remarks:} 1) Note that the weak convergence in $L^1$ of the determinant of minimizing sequences proves the lower semicontinuity of $E$ for minimizing sequences in $\mathcal{A}_{nc}^{axi}$ (cf.\, e.g.\ Theorem 1, p.12 in \cite{Cartesiancurrents2}). \\

2) Proposition \ref{detconv} is a direct consequence of Theorem 3 in \cite{HenaoMoraCorral2010}. However we use Theorem 2 in \cite{HenaoMoraCorral2010} to obtain that the limit $u$ is one-to-one a.e.\ and we provide an alternative proof of the fact that $\det Du_n \rightharpoonup \det Du$, by using a result of De-Lellis-Ghiraldin. This could be relevant in the future for the full problem of minimizing $E$ in $\mathcal{A}_{nc}$ (without assuming axisymmetry) because Theorem 2 of \cite{HenaoMoraCorral2010} can be used even without yet knowing if $(\cof Du_n)_n$ is equi-integrable. We first recall that result:

\begin{theorem}[De Lellis-Ghiraldin \cite{DeLellisGhiraldin}]
Let $m \geq 1$ and $\O\subset \R^m$ be an open set. Let $u\in L^q \cap W^{1,p}(\O,\R^m)$, with $p\geq m-1,$ $\frac 1q +\frac{m-1}{p}\leq 1$, be such that $\Det Du$ is a Radon measure. Let $\nu$ be the density of the absolutely continuous part $\Det Du$ with respect to the Lebesgue measure. Then $\nu(x)=\det Du(x)$ for $\mathcal{L}^m$-a.e.\ $x \in \O$.
\end{theorem}

\begin{proof}[Proof of Proposition \ref{detconv}]
First we note that we can apply Theorem 2 in \cite{HenaoMoraCorral2010} to obtain that the limit $u$ is one-to-one a.e. Now let $(u_n)_n$ be a minimizing sequence  for $E$ in $\mathcal{A}_{nc}^{axi}$. Since $\sup_n \int_\O H(\det Du_n) < +\infty$, using the De La Vallée Poussin criterion, we have that there exists $\theta \in L^1(\O)$ such that
\[ \det Du_n \rightharpoonup \theta \text{ in } L^1(\O). \]
Since $\det Du_n >0$ a.e.\ we obtain that $\theta \geq 0$ a.e. If $\theta$ were zero in a set $A$ of positive measure, then we would have (for a subsequence) $\det Du_n \rightarrow 0$ in $L^1(A)$ and almost everywhere in $A$. Because of the assumptions on $H$ we would have $H(\det Du_n) \rightarrow +\infty$ a.e.\ in $A$ and hence using Fatou's lemma we would obtain that $E(u_n)\rightarrow +\infty$, which is impossible. Hence $\theta>0$ a.e.\ in $\O$.
Now since $\mathcal{E}(u_n)=0$ for every $n$ we also have that $\Det Du_n=\det Du_n$ for every $n$, and hence
\[ \Det Du_n \rightharpoonup \theta \text{ in } \mathcal{D}'(\O).\]
But recall that
\[(\cof Du_n)^T.u_n=\frac{1}{r}(v_n\wedge \p_zv_n,0, -v_n\wedge \p_rv_n).\]
By using the same argument as in the end of the proof of Proposition \ref{cofconv}, we have that
\[(\cof Du_n)^T.u_n \rightharpoonup (\cof Du)^T.u \text{ in } L^2(K)\]
for every compact set $K\subset \O\setminus \{O_z\}$. Since $\Det Dw =\frac{1}{3}\dive[(\cof Dw)^T.w]$ for $w$ in $ W^{1,2}\cap L^\infty(\O,\R^3)$ we find that
\[\Det Du_n \rightharpoonup \Det Du \text{ in } \mathcal{D}'(\O \setminus \{O_z\}). \]
Thus $\Det Du=\theta$ on $ \O \setminus\{O_z\}$. But this means that $\Det Du$ is a Radon measure in $\O \setminus \{O_z\}$. We can then apply a result of De Lellis-Ghiraldin \cite{DeLellisGhiraldin} to state that the absolutely continuous part of the distributional Jacobian is equal to $\det Du$ a.e.\ in $\O \setminus \{O_z\}$:
\[\det Du =\theta >0 \text{ a.e.\ on }  \O \setminus \{O_z\}.\]
But since $\O\cap \{O_z\}$ has zero Lebesgue measure we find that $\theta=\det Du$ a.e.\ in $\O$ and
\[\det Du_n \rightharpoonup \det Du \text{ in } L^1(\O).\]
\end{proof}

\begin{proposition}\label{littlemain}
Let $\O\subset \R^3$ be a smooth bounded axisymmetric domain such that $$\inf_{(x,y,z)\in \O}\sqrt{x^2+y^2}>0.$$ Then
$\inf\{ E(w); w \in \mathcal{A}_{nc}^{axi}\}$ is attained.
\end{proposition}

\begin{proof}
Let $(u_n)_n$ be a minimizing sequence for $E$ in $\mathcal{A}_{nc}^{axi}$, we have seen that (up to a subsequence) $u_n \rightharpoonup u$ in $W^{1,2}$ for some $u \in W^{1,2}(\O,\R^3)$. By applying the previous propositions we have:
\[ \det Du_n \rightharpoonup \det Du \text{ in } L^1(\O)\]
\[\det Du>0 \text{ almost everywhere in } \O \]
\[u \text{ is one-to-one almost everywhere in } \O.\]
The weak $L^1$ convergence of the determinant implies
\begin{equation}
E(u)\leq \liminf_{n\rightarrow +\infty} E(u_n).
\end{equation}
The only thing which remains to check is that $u\in \mathcal{A}_{nc}^{axi}$. Since (up to a subsequence) we have pointwise convergence, it is true that $u$ is axisymmetric. We now prove that $\mathcal{E}(u)=0$. We proceed as in the proof of Theorem 3 of \cite{HenaoMoraCorral2010}. Let $f\in C^\infty_c(\O\times \R^3,\R^3)$ satisfy $\|f\|_\infty \leq 1$. Since we have that $\cof Du_n \rightharpoonup \cof Du$ in $L^1(\supp f)$ and $\det Du_n \rightharpoonup \det Du$ in $L^1(\supp f)$ we can apply Lemma \ref{Egoroff} below to prove that
\begin{eqnarray}
 \lim_{n\rightarrow +\infty} \mathcal{E}_{u_n}(f)&=&\int_\O \langle\cof D u(x), D_xf(x,u(x))\rangle+\det D u(x)\dive_yf(x,u(x)) dx \nonumber \\
 &=&\mathcal{E}_u(f).\nonumber
\end{eqnarray}
But since $\mathcal{E}_{u_n}(f)=0$ for all $n$ we obtain $\mathcal{E}_u(f)=0$. This is valid for all $f\in  C^\infty_c(\O\times \R^3,\R^3)$ satisfying $\|f\|_\infty \leq 1$ and hence $\mathcal{E}(u)=0$. In other words we do have that $u$ is in $\mathcal{A}_{nc}^{axi}$ and hence $u$ minimizes the energy $E$ in that space.
\end{proof}

In the previous proof we have used the following:

\begin{lemma}\label{Egoroff}[cf.\ Lemma 6.7 in \cite{SivaloganathanSpector} ]
Let $\psi_n\in L^\infty(\O)$ and $\theta_n \in L^1(\O)$ which satisfy
\[\psi_n \rightarrow \psi \text{ pointwise a.e. }\]
\[\theta_n \rightharpoonup \theta \text{ in } L^1(\O)\]
with $\|\psi_n\|_{L^\infty}\leq C$ for some $C>0, \theta \in L^1(\O)$ and $\psi \in L^\infty(\O)$. Then
\[\theta_n\psi_n \rightharpoonup \theta \psi \text{ in } L^1(\O).\]
\end{lemma}

The proof of Theorem \ref{main} follows from Proposition \ref{cofconv}, \ref{detconv} and \ref{littlemain}. \\

With the same method we used for Problem 1 we can treat Problems 2,1' and 2', Once we have the weak $L^1$ convergence of $\cof Du_n$ (given by Proposition \ref{cofconv}, which remains valid for the other problems) we can apply \cite[Theorem~3]{HenaoMoraCorral2010} to obtain the lower semicontinuity of the corresponding energy (see also \cite[Theorem~4.2]{MullerSpector1995}). However, note that the proof of Proposition \ref{detconv} does not require the equi-integrability of the cofactors.


\section{Minimizing sequences in $\mathcal{A}_{nc}^{axi}$ when the domain $\O$ contains the axis of rotation}

In the case where $\O$ is an axisymmetric domain such that $\inf_{(x,y,z)\in \O} \sqrt{x^2+y^2}>0$ we have seen that the main ingredient is to establish that if $(u_n)_n$ is a minimizing sequence for $E$ in $\mathcal{A}_{nc}^{axi}$ then
\[ \cof Du_n \rightharpoonup \cof Du \ \ \text{in } L^1(K) \ \ \text{ for all compact sets } K\subset \O.\]
As shown by the Conti-De Lellis' example, a priori this is not true when we consider an axisymmetric domain such that $\inf_{(x,y,z)\in \O}\sqrt{x^2+y^2}=0$. The energy could be concentrating on a lower dimensional set, which implies that minimizing sequences do not converge strongly in $W^{1,2}(\O,\R^3)$ and that the cofactors of this sequence are not equi-integrable. The aim of this section is to (partially) describe this concentration phenomenon. A first approach is to describe the \emph{defect measure}, that is the measure $\mu$ such that
\[ |Du_n|^2 \rightharpoonup |Du|^2+\mu \ \ \text{weakly in } \mathcal{M}(\O).\]
We observe that such a measure exists (since $|Du_n|^2$ is bounded in $L^1(\O)$) and that $u_n\rightarrow u$ strongly in $W^{1,2}$ if and only if $\mu=0$ in $\O$. A second approach consists in describing the lack of compactness of the problem through the theory of \emph{Cartesian currents} (see, e.g.\ \cite{CartesianCurrentsI}, \cite{Cartesiancurrents2} for previous applications of this theory to problems in elasticity). We now recall some definitions and properties of currents. In this section $N$ is an integer $N\geq 1$.

\begin{definition}
Let $U\subset \R^N$ be an open set, let $k\in \mathbb{N}$. We say that $T$ is a $k$-dimensional current in $U$ if $T$ is a linear form on the set of $C^\infty$ $k$-differential forms with compact support in $U$, denoted by $\mathcal{D}^k(U)$. We denote such a current by $T\in \mathcal{D}_k(U)$. \\
The boundary of a $k$-dimensional current is the $(k-1)$-dimensional current defined by
\[ \langle \p T,\omega \rangle = \langle T, d \omega \rangle \text{ for every } \omega \in \mathcal{D}^{k-1}(U) \]
and the boundary of a $0$-dimensional current is set to be equal to $0$. \\
The mass of a current $T$ is
\[ \mathbb{M}(T):=\sup\{ \langle T, \omega \rangle; \omega \in \mathcal{D}^k(U), |\omega| \leq 1 \}.\]
\end{definition}

\begin{definition}
A current is said to be normal if $T$ and $\p T$ have finite mass. A current is rectifiable if it can be written as
\[ \langle T,\omega \rangle =\int_{\mathcal{M}} \langle \omega(x),\tau(x) \rangle \theta(x)d\mathcal{H}^k(x) \]
where
\begin{itemize}
\item[*] $\mathcal{M}$ is a $k$-rectifiable set
\item[*] $\tau$ is a unit $k$-dimensional vector that spans $\text{Tan}(E,x)$ for $\mathcal{H}^{k}$-a.e.\ $x\in E$, (such a $\tau$ is called an orientation).
\item[*] $\theta$ is a real function, called the multiplicity, which satisfies $\int_E |\theta |d\mathcal{H}^{k} <+\infty$.
If $T$ is a rectifiable current we write $T=[E,\theta,\tau]$.
\end{itemize}
\end{definition}

\begin{definition}
A current is an integer multiplicity rectifiable current if it is a rectifiable current such that the multiplicity $\theta$ takes integer values.
\end{definition}

An example of an integer multiplicity current is the one given by the integration on the graph of a function. Let $\O \subset \R^n$ be a smooth bounded domain. We denote by $\mathcal{A}^1(\O,\R^n)$ the class of vector-valued maps $u:\O\rightarrow \R^n$ that are a.e.\ approximately differentiable and such that all the minors of the Jacobian matrix $Du$ are integrable. For $u \in \mathcal{A}^1$ we let
\[ M(Du)(x)=(e_1,Du(x).e_1)\wedge ...\wedge (e_n,Du(x).e_n) \]
where $\{e_i\}_{i=1,...,n}$ is the standard basis of $\R^n$, and here $\wedge$ denotes the exterior product (for the definition and properties we refer, e.g.\ to Section 2.1 of \cite{CartesianCurrentsI}). We also let
\[ \mathcal{G}_u=\{(x,u(x));x\in \mathcal{A}_u \} \]
where
\[\mathcal{A}_u=\{x \in \O; u \text{ is approximately differentiable in } x \}. \]
For $u \in \mathcal{A}^1$ we can define a current $G_u \in \mathcal{D}_n(\R^n\times\R^n)$ by
\[ \langle G_u, \omega \rangle = \int_{\R^n \times \R^n} \langle \xi, \omega \rangle d\mathcal{H}^n \lfloor{\mathcal{G}_u}\]
with $\xi=\frac{M(Du)(x)}{|MDu(x)|}$. We can show that $\mathcal{G}_u$ is a countably rectifiable set and then $G_u$ is an integer multiplicity rectifiable current. Besides the mass of this current is equal to
\[\mathbb{M}(G_u)=\mathcal{H}^n(\mathcal{G}_u)=\int_\O|MDu|dx.\]

\textbf{Remark}: If $u\in W^{1,n-1}(\O,\R^n)$ with $\det Du\in L^1(\O)$ then  $u\in \mathcal{A}^1$. \\

We now introduce the concept of stratification of differential forms and of currents.
\begin{definition}
Let $\omega$ be an $n$-differential form on $\R^n\times \R^n$, we can write \[\omega= \sum\limits_{\substack{\alpha, \beta  \\ |\alpha|+|\beta|=n}} f_{\alpha,\beta}dx_\alpha\wedge dy_\beta\] with $\alpha$ and $\beta$ some multi-indices. For every integer $h$ we then define
\[ \omega^{(h)}=\sum\limits_{\substack{|\alpha|+|\beta|=n \\|\beta|=h}} f_{\alpha,\beta}dx_\alpha \wedge dy_\beta. \]
Let $T$ be a current on $\R^n_x \times \R^n_y$ we define the $h$-stratum of $T$ by
\[ \langle (T)_h, \omega \rangle=\langle T, \omega^{(h)} \rangle. \]
\end{definition}

We can now make a link between the surface energy $\mathcal{E}(u)$ defined for $u$ in $W^{1,n-1}\cap L^\infty(\O,\R^n)$ such that $\det Du \in L^1(\O)$ and the theory of currents. Let $f\in C^\infty_c(\O\times\R^n,\R^n)$ we define an $(n-1)$-differential form by
\[ \omega_f(x,y)=\sum_{j=1}^n(-1)^{j-1}f^j(x,y)\hat{dy^j} \]
where $\hat{dy^j}=dy^1\wedge...\wedge dy^{j-1}\wedge dy^{j+1}\wedge... \wedge dy^n.$ Note that $|\omega_f(x,y)|=|f(x,y)|$ for all $x,y$ in $\O\times \R^3$. We can check that \[\mathcal{E}_u(f)=\langle \p G_u, \omega_f \rangle\] and since $\omega_f$ is a $(n-1)$-vertical form (i.e.\ a form which can be written as $\omega^{(n-1)}$) we have that
\[\mathcal{E}(u)=\mathbb{M}((\p G_u)_{n-1}).\]
We recall that if $u\in W^{1,p}$ then $(\p G_u)_h=0$ for all $h\leq p-1$ (this can be shown by approximation by smooth functions).
Thus we have that
\[ \mathcal{E}(u)=\mathbb{M}(\p G_u), \]
in particular if $u\in \mathcal{A}_{nc}$ then $\p G_u=0$.

\begin{definition}
Let $\O\subset \R^n$ be a smooth open bounded set in $\R^n$. We say that $T$ is a Cartesian current in $\O\times \R^n$ if
\begin{itemize}
\item[i)] $T$ is an integer multiplicity rectifiable current $T=[\mathcal{M},\theta, T]$;
\item[ii)] $\mathbb{M}(T) <+\infty$ and
\[ \| T\|_1:=\sup\{ \langle T,|y|\varphi(x,y)dx \rangle; \varphi \in C^1_c(\O \times\R^n), |\varphi|\leq 1\} <+\infty ;\]

\item[iii)] $T_\lfloor{dx_1\wedge...\wedge dx_n}$ is a positive Radon measure in $\O\times \R^n$ and $\pi_\sharp T=\llbracket \O \rrbracket $, with $\pi:\R^n_x\times\R^n_y \rightarrow \R^n_x$, $(x,y) \mapsto x$;
\item[iv)] $\p T_{\lfloor{\O\times \R^n}}=0$.
\end{itemize}
\end{definition}
\textbf{Remarks}: \begin{itemize}
\item[1)] If $T=G_u$ then $\|T\|_1=\|u\|_{L^1}$. Hence if $u\in \mathcal{A}_{nc}$, then $G_u$ is a Cartesian current.
\item[2)] For the definition of the push-forward of a current by a smooth function we refer to p.\ 132 of \cite{CartesianCurrentsI}. We have denoted by $\llbracket \O \rrbracket$ the current defined by integration on $\O$.
\end{itemize}
 The theory of Cartesian currents is well-suited for the weak convergence thanks to the following results.

\begin{theorem}\label{closureandstructure}[see \cite{CartesianCurrentsI}, \cite{Cartesiancurrents2}]
\begin{itemize}
\item[1)] (Th.1, p.386, of \cite{CartesianCurrentsI}) \\
The set of Cartesian currents $\text{cart}(\O\times \R^n)$ is closed under weak convergence in the sense of currents of sequences $(T_k)_k$ which satisfy that there exists some $C>0$ such that
    \[ \mathbb{M}(T_k)+\|T_k\|_1 \leq C \ \ \ \forall k .\]
\item[2)] (Th.1, p.392, of \cite{CartesianCurrentsI}) \\
For every $T$ in $\text{cart}(\O\times\R^n)$ there exists an a.e.\ approximately differentiable map $u_T:\O \rightarrow \R^n$ such that
    \[T=G_{u_T}+S_T \]
    where $G_{u_T}$ is the current given by the integration on the graph of $u_T$ and $S_T$ is a vertical current (here vertical means that $\pi_\sharp S =0$).

\item[3)] (Th.5, p.399 of \cite{CartesianCurrentsI})\\ If $T_k=G_{u_k}$ satisfies $T_k \rightharpoonup T:=G_{u_T}+S_T$ in the sense of currents and $\mathbb{M}(T_k)+\|T_k\|_1 \leq C$ for all $k$ and for some $C>0$, then
\[u_{T_k} \rightharpoonup u_T \text{ in } BV(\O) .\]
\end{itemize}
\end{theorem}

We can now prove that
\begin{proposition}\label{defectcurrent}
Let $\O\subset \R^3$ be a smooth bounded domain. Let $(u_n)_n$ be a minimizing sequence for $E$ in $\mathcal{A}_{nc}$. We suppose that $u_n\rightharpoonup u$ in $W^{1,2}(\O,\R^3)$ then
\[G_{u_n} \rightharpoonup G_u+S \]
with $\p G_u=-\p S$, and $S$ is a completely vertical integer multiplicity rectifiable current  with boundary $\p S$. Besides we have
\[ \langle S, \omega \rangle=0 \ \forall \ \omega \text{ such that } \omega=\omega^{(2)} \text{ and } \omega^{(2)}=d_y\eta\]
for $\eta$ a $2$-form with compact support (compare Proposition 1, p.149 of \cite{Cartesiancurrents2}).
\end{proposition}

\textbf{Remark}: $S$ completely vertical means that $S_{(0)}=S_{(1)}=S_{(3)}=0$ whereas $\p S$ completely vertical means $\langle \p S, \omega \rangle=0$ for every $\omega= \omega^{(1)}+\omega^{(2)} \in \mathcal{D}(\O\times\R^n)$.

\begin{proof}

Since $(u_n)_n$ is bounded in $W^{1,2}(\O,\R^3)$  and $\E(u_n)=0$ we have that $\mathbb{M}(G_{u_n})$ and $\mathbb{M}(\p G_{u_n})$ are uniformly bounded. Then from the compactness theorem for integral currents we have that $G_{u_n} \rightharpoonup T$ for some integral current $T$. But using that $\|G_{u_n}\|_1=\|u_n\|_{L^1}$ is also uniformly bounded we can apply 1) of Theorem \ref{closureandstructure} to obtain that $T$ is in $\text{cart} (\O\times \R^n)$. Now we apply the theorem of structure (2) and the convergence theorem (3) of Theorem \ref{closureandstructure} to decompose $T$ as $T=G_{u_T}+S$ for some $u_T$ in $BV(\O)$. But since we assumed that $u_n \rightharpoonup u$ in $W^{1,2}$ we have, from point 3) of Theorem \ref{closureandstructure}, that $u_T=u$. \\

Now let $\omega=\omega^{(0)}+\omega^{(1)}=\omega_{ij}dx_i\wedge dx_j+\beta_{ij}dx_i \wedge dy_j$, since $\langle G_{u_\e}, d\omega \rangle \rightharpoonup \langle G_u, d\omega \rangle$ for $(u_\e)_\e$ a sequence of smooth functions which approximate $u$ in $W^{1,2}(\O,\R^3)$, we obtain that $\langle \p G_u, \omega \rangle=0$. This proves that $\langle \p S, \omega \rangle =0$. \\
We then have
\begin{eqnarray}
\langle S, d_y (\omega_{ij}(x,y)dx^i\wedge dy^j)\rangle &= &\langle S, d[w(x,y)dx^i\wedge dy^j]-\langle S,d_x(\omega(x,y))dx^i\wedge dy^j \rangle \nonumber \\
&=& \langle \p S,\omega(x,y) dx^i\wedge dy^j \rangle-\langle S,\omega_{x_m} dx^m \wedge dx^i\wedge dy^j \rangle \nonumber \\
&=& 0 \nonumber
\end{eqnarray}
because $\p S$ and $S$ are completely vertical.
\end{proof}

We want to emphasize that if $u_n \rightharpoonup u$ in $W^{1,2}(\O,\R^3)$ and $\cof Du_n \rightharpoonup \cof Du$ weakly in $L^1(\O)$ then we have $G_{u_n}\rightharpoonup G_{u}$ in the sense of currents and then there is no defect current $S$ in Proposition \ref{defectcurrent}. This can be seen using Proposition 2 p.232 of \cite{CartesianCurrentsI}, and the proof of this fact relies on Lemma \ref{Egoroff}. This shows once more the important role of the equi-integrability of the cofactors of a minimizing sequence in these variational problems.
\begin{proposition}\label{currentconv}
Let $(u_n)_n$ be a minimizing sequence  for $E$ in $\mathcal{A}$. We suppose that $u_n \rightharpoonup u$ in $W^{1,2}(\O,\R^3)$.
Assume furthermore that $\cof Du_n \rightharpoonup \cof Du$ in $L^1(\O)$ then $G_{u_n}\rightharpoonup G_u$ in the sense of currents.
\end{proposition}

Note that Cartesian currents can be used to describe the problem of neo-Hookean materials in a non-cavitation setting (Problem 1 and Problem 1') but there also exists a class of currents adapted to the same problem where we allow for cavitation. This is the class of \emph{graphs} (cf.\ Definition 2, p.\ 385 in \cite{CartesianCurrentsI}). We have indicated at the end of Section 3 that we can solve Problem 2 (in axisymmetric domains such that $\inf_{(x,y,z)\in \O}\sqrt{x^2+y^2}>0$) by using the equi-integrability of the cofactors (given by Proposition \ref{cofconv}) and the weak $L^1$ continuity of the determinant (given by Theorem 3 in \cite{HenaoMoraCorral2010}). With the help of the currents theory we can also give an alternative proof of that fact. Before we recall the definition of the support of a current.
\begin{definition}
Let $T$ be a $k$-dimensional current in $\O\times \R^n$, the support of $T$ is the smallest closed set such that
$\langle T, \omega \rangle=0$ if the support of $\omega$ is contained in the complement of this closed set.
\[\supp T:= \cap \{ K\subset \O\times \R^3; K \text{ is relatively closed } in \ \O\times \R^n \]
\[\phantom{aaaaaaaaaaaaaaaaaaaaaaaaa} \text{ and } \langle T,\omega\rangle =0, \forall \omega \in \mathcal{D}^k(\O\times \R^n); \supp \omega \subset (\O\times \R^n)\setminus K \}.\]
\end{definition}

\begin{proposition}
Let $\O\subset \R^3$ be an axisymmetric domain such that $$\inf_{(x,y,z)\in \O}\sqrt{x^2+y^2}>0.$$ For $\lambda>0$ there exists a minimizer of $E(u)+\lambda\mathcal{E}(u)$ in the space $\mathcal{A}^{axi}$.
\end{proposition}

\begin{proof}(Sketch)
Let $(u_n)_n$ be a minimizing sequence for $E(u)+\lambda \mathcal{E}(u)$ in $\mathcal{A}^{axi}:=\{v\in \mathcal{A}; v \text{ is axisymmetric} \}$. We have that, up to the extraction of a subsequence, $u_n\rightharpoonup u$ in $W^{1,2}$. From Proposition \ref{cofconv} we have
\[ \cof Du_n \rightharpoonup \cof Du \text{ in } L^1(K) \ \text{for every }  \text{compact set } K\subset \O\setminus\{O_z\}. \]
We then have thanks to a slight adaptation of Proposition \ref{currentconv} that $G_{u_n}\rightharpoonup G_{u}$ in the sense of currents and then $\p G_{u_n}\rightharpoonup \p G_{u}$. We also have (see proof of Proposition \ref{detconv}) that $\det Du_n \rightharpoonup \theta$ weakly in $L^1(\O)$ with $\theta>0$, and $\Det Du_n \rightharpoonup \Det Du$ in $\mathcal{D}(\O)$. We now use a result of Mucci (Proposition 3.1 in \cite{Mucci}) which states that for every $u$ in $W^{1,2}\cap L^\infty(\O,\R^3)$ such that $\det Du\in L^1(\O)$ we have that for all $\varphi\in C^\infty_c(\O)$
\begin{equation}\label{Det=det}
\langle \Det Du,\varphi\rangle -\langle \det Du,\varphi \rangle=-\pi_\sharp[(\p G_u)_2\lfloor{\hat{\pi}^\sharp \omega_3}](g),
\end{equation}
with $\omega_3=\frac 13 \sum_{j=1}^3(-1)^{j-1}y^j\hat{dy^j}$, $\pi:\R^n\times \R^n \rightarrow \R^n, (x,y)\mapsto x$ and $\hat{\pi}:\R^n\times\R^n \rightarrow \R^n, (x,y) \mapsto y$ and for the push-forward (or image) operation of a current by a function we refer to p.132 of \cite{CartesianCurrentsI}. We would like to pass to the limit in the previous expression. For general $f$ ($f$ smooth), $f_\sharp$ is not continuous for the convergence of currents as shown by an example p.132 of \cite{CartesianCurrentsI}. But in our situation, since we assume that $\|u_n\|_\infty \leq M$, the support of $G_{u_n}$ cannot ``go to infinity". More precisely we have that, from the definition:
\[ \pi_\sharp[(\p G_{u_n})_2\lfloor{\hat{\pi}^\sharp \omega_3}](g)= \langle (\p G_{u_n})_2,\zeta_n \pi^\sharp(g\hat{\pi}^\sharp \omega_3) \rangle\]
with $\zeta_n\in C^\infty_c(\O\times\R^3)$ such that $\zeta_n=1$ on $\supp G_{u_n} \cap \supp \pi^\sharp(g\hat{\pi}^\sharp \omega_3)$. However, since $\|u_n\|_\infty \leq M$ we can choose $\zeta$ independent of $n$, we can take $\zeta\in C^\infty_c(\O\times\R^3)$ such that $\zeta \equiv 1$ on $\bigcap_{n\in \mathbb{N}} \supp G_{u_n} \cap \supp \pi^\sharp(g\hat{\pi}^\sharp \omega_3)$ for each $g\in C^\infty_c(\O)$. This proves the desired continuity in our special case. Thus passing to the limit in \eqref{Det=det} we obtain that
\begin{equation}\nonumber
\langle \Det Du,\varphi\rangle -\langle \theta,\varphi \rangle=-\pi_\sharp[(\p G_u)_2\lfloor{\hat{\pi}^\sharp \omega_3}](g),
\end{equation}
for every $g\in C^\infty_c(\O)$. But applying the identity \eqref{Det=det} again we obtain that $\det Du=\theta>0$. Thus applying Theorem 2 in \cite{HenaoMoraCorral2010} to obtain that the limit $u$ is one-to-one almost everywhere, we have that $u\in \mathcal{A}^{axi}$. Besides the energy is lower semicontinuous for the sequence $u_n$ thanks to the weak $L^1$ continuity of the determinant for the $E(u)$ part and thanks to an argument similar to the one used at the end of the proof of Proposition \ref{detconv} for the $\mathcal{E}(u)$ part. This concludes the proof.
\end{proof}

If we assume that $\O$ is axisymmetric we can say more about the defect current $S$, we can describe its support.
\begin{proposition}
Let $\O \subset\R^3$ be an axisymmetric domain. Let $(u_n)_n$ be a minimizing sequence for $E$ in $\mathcal{A}_{nc}^{axi}$. We assume that $u_n \rightharpoonup u$ in $W^{1,2}$ then $G_{u_n}\rightharpoonup G_u+S$ and $\supp S \subset \left(\O\cap \{O_z\} \right) \times \R^3$.
\end{proposition}
\begin{proof}
We have seen (cf.\ Proposition \ref{cofconv}) that $\cof Du_n \rightharpoonup \cof Du$ in $L^1(K)$ for every compact set $K\subset \O \setminus \{O_z\}$. We have also, from Proposition \ref{detconv}, that $\det Du_n \rightharpoonup \det Du$ in $L^1(\O)$. The other minors of the Jacobian matrix (actually the coefficients of this matrix) converge weakly in $L^1(\O)$ to the limiting coefficients (because $Du_n \rightharpoonup Du$ in $L^2(\O,\R^3)$). Thus we have
\[ \langle G_{u_n}, \omega \rangle \rightarrow \langle G_u, \omega\rangle \]
for all $\omega\in \mathcal{D}(\O\times \R^n)$ such that \[\supp \omega \subset \O\setminus \{O_z\} \times \R^3.\] This concludes the proof.
\end{proof}

We have thus an information on the lack of compactness of the problem in the axisymmetric case from the point of view of currents. There is a relation between the defect measure and the defect current.

\begin{proposition}
Let $(u_n)_n$ be a minimizing sequence in $\mathcal{A}_{nc}$ for $E$. We assume that $u_n \rightharpoonup u$ in $W^{1,2}$, $G_{u_n} \rightharpoonup G_u+S$ in the sense of currents and $|Du_n|^2 \rightharpoonup |Du|^2+\mu$ in the sense of measures then:
\[ \supp S \subseteq \supp \mu \times \R^3. \]
\end{proposition}

\begin{proof}
Let $\omega\in \mathcal{D}^3(\O\times \R^3)$ be such that $\supp \omega \subset (\O\setminus\supp \mu)\times \R^3$. We claim that $\langle G_{u_n}, \omega \rangle \rightarrow \langle G_u, \omega \rangle$ because $u_n \rightarrow u$ in $W^{1,2}_{\text{loc}}(\O \setminus \supp \mu)$. Indeed for every compact set $K\subset \O \setminus \supp \mu$ we have
\[ \limsup_{n\rightarrow +\infty} \int_K |Du_n|^2 \leq \int_K |Du|^2 \]
because $|Du_n|^2 \rightharpoonup |Du|^2$ in the sense of measures on $\O \setminus \supp \mu$. On the other hand we always have $\int_K |Du|^2 \leq \liminf_{n\rightarrow +\infty} \int_K |Du_n|^2$. We thus have that $$\lim_{n\rightarrow +\infty} \int_K |Du_n|^2= \int_{K} |Du|^2$$ and this implies the strong convergence of $u_n$ in $W^{1,2}_{\text{loc}}(\O \setminus \supp \mu)$. Hence $\langle S, \omega \rangle =0$ for every $\omega \in \mathcal{D}^3((\Omega\setminus \supp \mu)\times \R^3)$ which means that $\supp S \subset \supp \mu \times \R^3$.
\end{proof}

We have already mentioned that the example of Conti-De Lellis corresponds to a dipole phenomenon in harmonic maps theory. This phenomenon appears, for example if we consider a sequence of stationary harmonic maps from a domain $\O\subset \R^n$, $n\geq3$ into a smooth compact manifold with  uniformly bounded energies. This was studied by Lin in \cite{Linstationary} and Lin-Rivi\`ere in \cite{LinRiviere}. In that case they were able to give a precise description of the defect measure, it is supported on a rectifiable set of Hausdorff dimension $(n-1)$. In dimension $n=3$ and when the target manifold is the sphere $\mathbb{S}^2$ the lack of compactness of sequence of maps with uniformly bounded energies can be analyzed with the theory of Cartesian currents. This has been done in \cite{GiaquintaModicaSoucek} and \cite{Cartesiancurrents2}. In that case they proved that the defect current is a one-dimensional rectifiable current times $\llbracket \mathbb{S}^2\rrbracket$ (the integration on the sphere). This is another description of the dipole. Note that axisymmetric harmonic maps have been studied in \cite{HardtLinPoon} and \cite{Martinazzi}. In those cases also, the singularities are located only on the axis of rotation of the domain.

\section{ $W^{1,2}$-quasiconvexity of the neo-Hookean energy}

The aim of this section is to consider the case where $\O \subset \R^3$ is not necessarily axisymmetric but the boundary condition is affine, i.e. $u_{|\p \O}=\lambda x+b$ for some $\lambda \in \R$ and $b\in \R^3$. If we work in a class where no cavitation  can occur, for example in the class $\mathcal{A}_{nc}$, we expect $\int_\O \det Dudx$ to be a null-Lagrangian and, hence, the affine transformation $u(x)=\lambda x+b$ to be a minimizer (using Jensen's inequality). The heuristic argument is the following. Recall that if $u$ is in $\mathcal{A}_{nc}$ then $\det Du=\Det Du=\frac 13 \dive [(\cof Du)^T.u]$. From the Jensen's inequality we have:
\begin{equation}\nonumber
\frac{1}{|\O|}\int_\O H(\det Du) \geq H\left( \frac{1}{|\O|}\int_\O \det Du \right).
\end{equation}
Now we can think that
\begin{eqnarray}\label{falseproof}
\int_\O \det Du &=&\int_\O \frac 13 \dive [(\cof Du)^T.u] \nonumber \\
&=& \int_{\p \O} \frac 13 (\cof Du)^T.u \cdot \nu \\
&=& \frac 13 \int_{ \p \O} \lambda^3 x\cdot \nu +\lambda^2b\cdot \nu \nonumber \\
&=& \lambda^3 |\O|. \nonumber
\end{eqnarray}
Then we could infer that $\int_\O H(\det Du) \geq H(\lambda^3)|\O|$ with equality if $u$ is affine and a similar treatment of the Dirichlet part of the energy would yield that the affine transformation is a minimizer. However the second inequality in \eqref{falseproof} is not true in general. Indeed M\"uller-Spector built the following example in \cite{MullerSpector1995}. Let $C=(-1,1)\times (0,1)$ they constructed a map $f\in W^{1,p}(C,\R^2)$, for every $1\leq p<2$ such that $\Det Df=\det Df$, $f_{|\p C}= \text{Id}$ and $\int_C \det Df \neq |C|$. What happens is that $f$ opens a cavity at the boundary of the domain, which is not detected by the distributional determinant. To avoid this phenomenon we work in a new class where both cavitation in the interior and at the boundary are excluded. Following Sivaloganathan-Spector, \cite{SivaloganathanSpector}, we avoid this possibility of cavitation at the boundary by working in a new admissible class where the condition $\mathcal{E}(u)=0$ is imposed in a domain slightly larger than $\O$. The example of M\"uller-Spector shows the failure of the divergence formula in \eqref{falseproof} when it is understood in its classical pointwise. However Chen-Frid in \cite{ChenFrid1999} (see also \cite{ChenTorresZiemer2009}) generalized the divergence formula for weakly differentiable vector fields ($(\cof Du)^T.u$ satisfies the hypothesis in \cite{ChenFrid1999}) by interpreting the normal trace as some measure on $\p \O$. If we apply that theory to a map $u$ which exhibits cavitation at the boundary we can show that the measure in the theory of Chen-Frid will be equal to the classical normal trace plus a sum of Dirac masses centered at the cavitation points (each Dirac mass being multiplied by a coefficient equal to the volume of the cavity).   \\

We choose $\tilde{\O}$ a smooth bounded domain such that $\O \subset \subset \tilde{\O} \subset \R^3$. We let $g(x)=\lambda x+ b$ for some $\lambda\in \R$, $b\in \R^3$. Let $u$ be in $\mathcal{A}_{nc}$, by definition we have $u_{|\p \O}=g$. We define an extension of $u$ in the following way: we let $u^e$ be defined by
\begin{equation}\nonumber
u^e=\begin{cases} u(x) \text{ for } x\in \O \\
g(x) \text{ for } x\in \tilde{\O}\setminus \O.
\end{cases}
\end{equation}
Note that since $u_{|\p \O}=g$ we have $u^e\in W^{1,2}(\tilde{\O},\R^3)$. We then set
\begin{equation}\nonumber
\mathcal{A}_{ncb}=\{v \in \mathcal{A} ; \  \mathcal{E}_{\tilde{\O}}(v^e)=0\}
\end{equation}
where
\begin{multline}
\mathcal{E}_{\tilde{\O}}(w)=\sup \big\{ \int_{\tilde{\O}} \langle \cof Dw, D_xf(x,w(x))\rangle +\det Dw \dive_y f(x,w(x)) ; \\ f \in C^1_c(\tilde{\O}\times \R^3,\R^3), \ \|f\|_\infty \leq 1 \big\}. \nonumber
\end{multline}
The difference with the definition of the surface energy $\mathcal{E}$ in $\O$ is that here the test functions are defined in the larger domain $\tilde{\O}$.

Contrarily to the situation when cavitation is allowed, in which the argument of \cite{Ball2009} p.\ 10 suggests that minimizers might not exist, we do have existence of minimizers if some affine data is prescribed on the boundary and if we exclude the possibility of cavitation in the interior and at the boundary.

\begin{proposition}
There exists a minimizer of $E$ in $\mathcal{A}_{ncb}$. Furthermore it is unique and equals $u(x)=\lambda x+b$.
\end{proposition}

\begin{proof}
Let $u\in \mathcal{A}_{ncb}$, since $H$ is convex Jensen's inequality yields
\begin{equation}\nonumber
\frac{1}{|\O|}\int_\O H(\det Du) \geq H\left(\frac{1}{|\O|}\int_\O \det Du\right).
\end{equation}
Since $u\in \mathcal{A}_{ncb}$ (in particular $u\in \mathcal{A}_{nc}$) we have that $\det Du=\Det Du= \frac 13 \dive\left[ (\cof Du)^T.u \right]$. We now show that for $u\in \mathcal{A}_{ncb}$ we have
\begin{equation}\nonumber
\int_\O \det Du=\lambda^3|\O|.
\end{equation}
For that we let $F:=\frac 13 (\cof {Du^e})^T.u^e$. We have that $F$ is in $L^1(\tilde{\O},\R^3)$ and $\dive F=\det Du^e$ is also in $L^1(\tilde{\O},\R^3)$ (since $\mathcal{E}_{\tilde{\O}}(u^e)=0$). We then claim that we can find a sequence of vector fields $(F_j)_j$ such that
\begin{itemize}
\item[*] $F_j\in C^\infty(\tilde{\O},\R^3)$
\item[*] $F_j \rightarrow F$ strongly in $L^1(\tilde{\O},\R^3)$
\item[*] $\int_{\tilde{\O}} |\dive F_j| \rightarrow \int_{\tilde{\O}} |\dive F|$
\item[*] $\dive F_j \rightharpoonup \dive F$ weakly in $\mathcal{M}(\tilde{\O})$.
\end{itemize}

For a proof of this fact we refer to \cite[Theorem~1.2]{ChenFrid1999} (for the first three items) and the fourth item is an adaptation of \cite[Theorem~3, p.175]{EvansGariepy}.
Now we consider exterior deformations of $\p \O$. More precisely since we assumed that $\O$ is a non-empty smooth ($C^\infty$) open set, we can consider the exterior normal to $\p \O$, $\nu:\p \O \rightarrow \R^3$  and we have (cf.\ \cite[Theorem 16.25.2]{Dieudonne}):
\begin{itemize}
\item[i)] There exists $\delta>0$ such that the map $w:\p \O \times (-\delta,\delta) \rightarrow \tilde{\O}$ given by
\[ w(x,t)=x+t\nu(x) \ \ \ \ \ x\in \p \O, \ \ t\in \R \]
is a $C^\infty$-diffeomorphism from $\p \O \times (-\delta, \delta)$ into $$N(\p \O,\delta):=\{x\in \tilde{\O}, \dist(x,\p \O)<\delta\}.$$
\item[ii)] The function $d:\tilde{\O}\rightarrow \R$ given by
\[ d(x):= \begin{cases} -\dist(x,\p \O) \ \text{ if } x\in \O, \\
0 \ \text{ if } x\in \p \O, \\
+\dist(x,\p \O) \ \text{ if } x \in \tilde{\O} \setminus \O,
\end{cases}\]
is continuous in $\tilde{\O}$ and of class $C^\infty$ in $N(\p \O,\delta)$.
\item[iii)] For every $t$ in $(-\delta,\delta)$ the set $\O_t:=\{x \in \O; \dist(x,\p \O)<t\}$ is open and compactly contained in $\tilde{\O}$, and has $C^\infty$ boundary. We let $\nu_t$ be the exterior normal to $\p \O_t$.
\end{itemize}

Thus, for every $\varphi$ in $C^1_c(\R^n)$, thanks to the classical divergence formula we have:
\begin{equation}\label{classikdiv}
\int_{\O_t} \varphi \dive F_j= \int_{\p \O_t}\varphi F_j\cdot \nu_t-\int_{\O_t} F_j\cdot \nabla \varphi.
\end{equation}
Furthermore since $F=\lambda^3x+\lambda^2 b$ in a neighborhood of $\p \O_t$ we have that $F_j \rightarrow F$ uniformly in this neighborhood (the approximation $F_j$ is locally given by a convolution product with a regularizing kernel). Thus we have all the ingredients to pass to the limit as $j\rightarrow +\infty$ in \eqref{classikdiv} to obtain
\begin{eqnarray}\label{prepasst}
\int_{\O_t} \varphi \dive F &=& \int_{\p \O_t}\varphi F\cdot \nu_t-\int_{\O_t} F\cdot \nabla \varphi \nonumber \\
&=& \frac 13 \int_{\p \O_t} \lambda^2\varphi(\lambda x\cdot \nu_t+b\cdot \nu_t)-\int_{\O_t}F\cdot \nabla \varphi.
\end{eqnarray}
By using Lebesgue's dominated convergence theorem we have
\begin{itemize}
\item[*] $\int_{\O_t} \varphi \dive F \rightarrow \int_{\O}\varphi \dive F=\int_\O \varphi \det Du$ as $t\rightarrow 0$
\item[*] $\int_{\O_t} F\cdot \nabla \varphi \rightarrow \int_{\O} F\cdot \nabla \varphi$ as $t \rightarrow 0$.
\end{itemize}
Now observe that $\nu_t=\nu$ (where $\nu$ is the outer unit normal to $\p \O$) and by using the change of variable $x=y+t\nu(y)$ for $y\in \p \O$ we have that
\begin{eqnarray}
\frac 13 \int_{\p \O_t} \lambda^2\varphi(\lambda x\cdot \nu_t+b\cdot \nu_t) = \frac 13 \int_{\p \O} \lambda^2 \varphi(y+t\nu(y))\left[ \lambda (y+t\nu(y))\cdot\nu(y)+b\cdot \nu(y)\right]. \nonumber
\end{eqnarray}
Applying Lebesgue's dominated convergence theorem again we obtain that
\[ \int_{\p \O_t} \lambda^2\varphi(\lambda x\cdot \nu_t+b\cdot \nu_t) \rightarrow \int_{\p \O} \lambda^2\varphi(\lambda x\cdot \nu+b\cdot \nu). \]
We can then take $\varphi \equiv 1$ on $\overline{\O}$, and we observe that $\int_{\p \O} b\cdot \nu=0$ for a constant vector $b$ so we deduce that
\begin{equation}\nonumber
\int_\O \det Du=\frac{\lambda^3}{3} \int_{\p \O} x\cdot \nu= \lambda^3|\O|.
\end{equation}
For the last equality we used the divergence formula. Thus we have that, for every $u$ in $\mathcal{A}_{ncb}$
\begin{equation}
\int_\O H(\det Du) \geq |\O| H(\lambda^3)
\end{equation}
with equality for $u=\lambda x+b$ in $\O$.

%
%

On the other hand, we also have that
\[\int_\O |Du|^2=\sum_{i,j=1}^3 \int_\O|\p_iu^j|^2 \]
\begin{eqnarray}\label{jensen2}
\frac{1}{|\O|}\int_\O|Du|^2 &\geq & \sum_{i,j=1}^3 \left( \frac{1}{|\O|}\int_\O \p_iu^j\right)^2 \nonumber \\
& \geq & \sum_{i,j=1}^3 \left( \int_{\p \O} u_j\cdot \nu_i \right)^2 \nonumber \\
& \geq & \sum_{i,j=1}^3 \left( \frac{1}{|\O|}\int_{\p \O} \lambda x_j\nu_i\right)^2
\end{eqnarray}
and again there is equality for $u=\lambda x+b$ on $\O$ (this times observe that $\int_{\p \O} x_j\nu_i=\delta_{ij}|\O|$ thanks to the divergence formula).  This proves that $u=\lambda x+b$ is a minimizer of $E$ in $\mathcal{A}_{nc}$. Note that since $x \mapsto x^2$ is strictly convex, thanks to the equality case in Jensen's equality, we have equality in \eqref{jensen2} if and only if $Du$ is a constant matrix. By using the boundary condition we then find that $u(x)=\lambda x +b$ is the unique minimizer of $E$ in $\mathcal{A}_{nc}$.
\end{proof}

The previous proposition shows that the lack of compactness of the problem does not prevent the existence of minimizers (at least in some cases). We note that the example of Conti-De Lellis is local and can be built for any boundary data $g$. Indeed they constructed their example in a ball, and by rescaling we can adapt their example to any bounded domain $\O$. Thus their example is not necessarily an obstruction to the existence of minimizers.

\section{Equations satisfied by a minimizer $u$ in $\mathcal{A}_{nc}^{axi}$}

 Now that we have obtained the existence of a minimizer $u$ of $E$ in the set $\mathcal{A}_{nc}^{axi}$ it is natural to wonder if such a minimizer satisfies some equations. We note that if we can build some variations
\begin{eqnarray}
(-\e,\e) & \rightarrow & \mathcal{A}_{nc}^{axi} \nonumber \\
t & \mapsto u_t \nonumber
\end{eqnarray}
with $u_{t=0}=u$ a minimizer of $E$ in $\mathcal{A}_{nc}^{axi}$ and such that $t\mapsto E(u_t)$ is differentiable then we will obtain $\frac{d}{dt}_{|t=0}E(u_t)=0$. First we can think of variations of the form $u_t=u+t\Phi$ with $\Phi \in C^\infty_c(\O,\R^3)$. However it is difficult to prove that $u_t$ is in $\mathcal{A}_{nc}^{axi}$ for $t$ small. It can happen that $\det(Du_t)=0$ on a set of positive Lebesgue measure. These variations would lead us to some Euler-Lagrange equations for $u$ and we note that it is an open problem in elasticity to know if a minimizer does satisfy the Euler-Lagrange equations associated to its functional (cf.\ \cite{Ballmin}). \\

Instead of the previously mentioned variations we will consider some special inner variations. For $u$ axisymmetric there exists $v: \O_0 \rightarrow \R^2$ such that $u(r\cos \theta,r\sin \theta,z)=v_1(r,z)e_r+v_2(r,z)e_z$, with $v_1(r,z)\geq 0$ a.e.\ in $\O_0$. We let $X=(r,z)$ and \[v_t(X)=v(X+t\varphi(X)) \text{ for } t\in (-1,1).\] We will show that $v$ satisfies some equations and so does $u$. Furthermore these equations are the same that we would expect if $u$ were a minimizer of $E$ in $\mathcal{A}_{nc}$. We recall that for $u$ in $\mathcal{A}_{nc}^{axi}$ with $u=v_1 e_r+v_2e_z$ we have $E(u)=2\pi G(v)$ with
\[G(v)=\int_{\O_0} |D  v|^2rdrdz +\int_{\O_0}\frac{v_1^2}{r}drdz+\int_{\O_0}H\left(\frac{v_1}{r}\det D v\right)rdrdz. \]

In this section we need some supplementary hypotheses on $H$, these are inspired by \cite{BaumanOwenPhillips}.
There exist  $s,c_1,c_2,d_0>0$ such that
\begin{equation}\label{hypothesisH1}
c_1 t^{-s-k} \leq (-1)^k\frac{d^k}{dt^k} H(t) \leq c_2 t^{-s-k} \text{ for } k=0,1 \text{ and for } t<d_0,
\end{equation}
and there exists $\tau,c_3,c_4,d_1>0$ such that
\begin{equation}
c_3t^{\tau+1} \leq H'(t) \leq c_4t^{\tau+1} \text{ for } t\geq d_1.
\end{equation}

\begin{theorem}\label{Equations}
Let $\O\subset \R^3$ be a smooth axially symmetric bounded domain such that $\inf_{(x,y,z)\in \O}\sqrt{x^2+y^2}>0$. Let $u$ be a minimizer of $E$ in $\mathcal{A}_{nc}^{axi}$. Let $v:\O_0\rightarrow \R^2$ such that $u(r\cos \theta,r\sin \theta,z)=v_1(r,z)e_r+v_2(r,z)e_z$. Then we have
\begin{align}\label{equation1}
\frac{1}{r}\p_r\left( r\left[(|\p_rv|^2-|\p_zv|^2-\frac{v_1^2}{r^2})+H'(\frac{v_1}{r}\det D v)\frac{v_1}{r}\det D v-H(\frac{v_1}{r}\det D v)\right] \right)+\nonumber \\
   \p_z(2\p_rv\cdot \p_zv) = \frac{1}{r}\left[ -|D v|^2+\frac{v_1^2}{r^2}+H'(\frac{v_1}{r}\det D v)\frac{v_1}{r}\det D v-H(\frac{v_1}{r}\det D v) \right]
\end{align}
\begin{align}\label{equation2}
\frac{1}{r}\p_r(2r\p_rv\cdot\p_zv)+ \nonumber \\
\p_z \left[ (|\p_zv|^2-|\p_rv|^2-\frac{v_1^2}{r^2})+H'(\frac{v_1}{r}\det D v)\frac{v_1}{r}\det D v -H(\frac{v_1}{r}\det D v) \right]=0 \end{align}
in the sense of distributions. This is equivalent to
\begin{equation}
\Dive\left( 2Du^TDu+\left[H'(\det Du)\det Du-|Du|^2-H(\det Du)\right]I \right)=0
\end{equation}
in the sense of distributions.
\end{theorem}

\textbf{Remark}: The last equation means that the energy-momentum tensor of $u$ associated to the functional $E$ is divergence free. That is what we would expect for a minimizer of $E$ in $\mathcal{A}_{nc}$. Indeed if $u$ is a minimizer of $E$ in $\mathcal{A}_{nc}$, we can prove as in \cite{BaumanOwenPhillips} that $\frac{d}{dt}_{|t=0}E(u_t)=0$ with $u_t=u(x+t\varphi(x))$ for some $\varphi:\O \rightarrow \R^3$ with compact support. The latter critical condition leads to the energy-momentum tensor being divergence free.  \\

The rest of this section is devoted to the proof of Theorem \ref{Equations}. For the comfort of the reader we divide that proof in several steps.

\begin{lemma}\label{appartenance}
Let $u\in \mathcal{A}_{nc}^{axi}$ and $v:\O_0\rightarrow \R^2$ be such that $$u(r\cos \theta,r \sin \theta,z)=v_1(r,z)e_r+v_2(r,z)e_z.$$ For each $\varphi\in C^1_c(\O_0,\R^2)$ there exists $\e_0>0$ such that $v_t(X):=v(X+t\varphi(X))$ satisfies that $u_t=v_t^1e_r+v_t^2e_z$ is in $\mathcal{A}_{nc}^{axi}$ for $|t|<\e_0$.
\end{lemma}

\begin{proof}
Let $v_t$ and $u_t$ be as in the statement of the lemma. It is clear that $u_t$ is axially symmetric and $v_t^1\geq 0$ a.e. For $t$ small enough $X+t\varphi(X)$ is a $C^1$-diffeomorphism of $\overline{\O}_0$ thus $v_t$ is one-to-one a.e.\ in $\O_0$ since $v$ is one-to-one a.e.\ in $\O_0$. This also means that $u_t$ is one-to-one a.e.\ in $\O$. We also have that $\|v_t\|_{L^\infty}=\|v\|_{L^\infty}$ and then $\|u_t\|_{L^\infty} \leq M$. \\
It holds that $D v_t(X)=D v(X+t\varphi(X))(I+tD\varphi(X))$ and thus $\det D v_t >0$ a.e.\ in $\O_0$ for $t$ small enough. In particular we have $\det Du_t >0$ a.e.\ for $t$ small enough.\\

We now check that $\mathcal{E}(u_t)=\E(u)=0$. More generally we prove that for every family of diffeomorphisms $\theta_t:\O \rightarrow \O$ with $t$ small enough  and such that $\theta_0=\text{Id}$, we have $\E(u\circ\theta_t)=\E(u)$. This would prove the result since we can write $u_t=u\circ\theta_t$ for the following family of diffeomorphisms
\begin{eqnarray}
\theta_t: \O & \rightarrow & \O \nonumber \\
(r,z,\theta) & \mapsto & (r+t\varphi_1,z+t\varphi_2,\theta) \nonumber
\end{eqnarray}
Now let $f\in C^1_c(\O\times \R^3,\R^3)$,
\begin{eqnarray}
\E_{u\circ \theta_t}(f) &=&\int_\O \langle \cof D(u\circ \theta_t), D_xf(x,u\circ\theta_t(x))\rangle+\det D(u\circ\theta_t)\dive_yf(x,u\circ\theta_t(x))dx \nonumber \\
&=&\int_\O \langle [\cof Du(\theta_t)\cof D\theta_t], D_xf(x,u(\theta_t))\rangle+ \nonumber\\
 & & \phantom{aaaaaaaaaaaaaaaaaaaaaaaaaa} \int_\O \det Du(\theta_t)\det D\theta_t \dive_yf(x,u(\theta_t))dx. \nonumber
\end{eqnarray}
We make the following change of variables: $z=\theta_t(x)$ and we let $f_t(z,y):=f(\theta_t^{-1}(z),y)$. We note that $\det D\theta_t(x) >0$ a.e.\ for $t$ small enough, and we also observe that
\begin{eqnarray}
D_z[f_t(z,y)]&=&D_xf(\theta_t^{-1}(z),y)D\theta_t^{-1}(z) \nonumber \\
&=& D_xf(\theta_t^{-1}(z),y)\frac{\cof D\theta_t^T}{\det D\theta_t}. \nonumber
\end{eqnarray}
We thus obtain that

\begin{eqnarray}
\E_{u\circ\theta_t}(f) &=& \int_\O \langle \cof Du(z), D_xf(\theta_t^{-1}(z),u(z))\rangle+\det Du(z)\dive_yf(\theta_t^{-1}(z),u(z))dz \nonumber \\
&=& \E_u(f_t). \nonumber
\end{eqnarray}
From that it follows that $\E(u\circ\theta_t)=\E(u)=0$.
\end{proof}
It remains to show that for the variations we are considering $t\mapsto E(u_t)$ is differentiable at $t=0$ and we need to compute the derivative. Since $E(u)=2\pi G(v)$ we will show that $t\mapsto G(v_t)$ is differentiable at $t=0$ and compute its derivative. We proceed in two steps, first we deal with the potential term of the energy.

To prove that the term $A(v)= \int_{\O_0} H(\frac{v_1}{r}\det D v) rdrdz$ is differentiable with respect to the type of variations we are considering we first establish an abstract lemma giving conditions under which a functional is differentiable for inner variations and then we check that $A$ satisfies these conditions. The next lemma is very close to Theorem A.1 of \cite{BaumanOwenPhillips}, we only need minor modifications to treat the case where the integrand of the functional is not autonomous (i.e.\ depends on $x$ and also on $u$). Before stating the lemma we need some notations. We let $$\mathcal{B}=\{v:\O_0\rightarrow \R^2; \ u(r\cos\theta,r\sin\theta,z)=v_1(r,z)e_r+v_2(r,z)e_z \in \mathcal{A}_{nc}^{axi} \}.$$ Let $\gamma:\O_0\times \R^2 \times M_2^+(\R) \rightarrow \R$ satisfy:
\begin{equation}\nonumber
\gamma \geq 0 \text{ on } \O_0\times \R^2 \times M_2^+(\R).
\end{equation}
\begin{equation}\nonumber
\gamma \in C^1(\O_0\times \R^2 \times M_2^+(\R)).
\end{equation}
There exist  $\theta>0, N>0$  such that  if  $|X-Y|<\theta$ and  $|C-I|<\theta$ then
\begin{eqnarray}
|D_X\gamma(Y,p,FC)|&\leq& N (1+|F|^2+\gamma(X,p,F)) \text{ and } \nonumber \\
|F^TD_F\gamma(Y,p,FC)| &\leq &N(1+|F|^2+\gamma(X,p,F)). \nonumber
\end{eqnarray}
We let \[W(v)=\int_{\O_0} \gamma(X,v,Dv)dX.\]

\begin{lemma}\label{preconv1}
Let $W,\gamma$ as before and  $v\in \mathcal{B}$ such that $W(v)<+\infty$ . For every $\varphi \in C^1_c(\O_0,\R^2)$, there exists $\e_0>0$ such that

\begin{itemize}
\item[i)] $v_t(X):=v(X+t\varphi(X))$ is in $\mathcal{B}$ for $|t|<\e_0$,
\item[ii)] $\frac{d}{dt}_{|t=0} W(v_t)$ exists and is equal to
\begin{eqnarray}\label{formuladerive}\nonumber
\frac{d}{dt}_{|t=0} W(v_t)&=&- \int_{\O_0} D_X\gamma(X,v,Dv)\cdot\varphi \\
&+&\langle Dv^TD_F\gamma(X,v,Dv),D\varphi\rangle-\gamma(X,v,Dv)\tr D\varphi dX
\end{eqnarray}
\end{itemize}
\end{lemma}

\begin{proof}
We prove that we can pass to the limit in the quotient $\frac{1}{t}(W(v_t)-W(v))$ by using the dominated convergence. We have
\begin{eqnarray}
\frac{1}{t}(W(v_t)-W(v)) = \frac{1}{t}\int_{\O_0}[\gamma(X,v_t,Dv_t)-\gamma(X,v,Dv)]dX \nonumber \\
 = \frac{1}{t}\left[\int_{\O_0} \gamma(Y-t\varphi,v(y),Dv(Y)(I+tD\varphi))\det (I+tD\varphi)^{-1}dY-\gamma(Y,v,Dv)dY \right] \nonumber \\
 =\int_{\O_0}\frac{1}{t}\left[ \gamma(Y-t\varphi,v,Dv(Y)(I+tD\varphi))-\gamma(Y,v,Dv)\right]\det(I+tD\varphi)^{-1}dY \nonumber \\
   \phantom{aaaa} +\int_{\O_0} \gamma(Y,v,Dv)\frac{1}{t}(\det(I+tD\varphi)^{-1}-1)dY \nonumber
\end{eqnarray}
We claim that, for $t$ small enough, we have
\begin{equation}\label{preconvdom}
|\gamma(Y-t\varphi(X),v,Dv(Y)(I+tD\varphi(x))-\gamma(Y,v,Dv)|\leq Nt[1+|Dv|^2+\gamma(Y,v,Dv)].
\end{equation}
We can then apply the dominated convergence theorem for the first term of the r.h.s of the last equality since for $t$ small enough we have $\frac{1}{4}\leq \det(I+tD\varphi)^{-1} \leq 4$ and that
\begin{eqnarray}
\lim_{t\rightarrow 0} \frac{1}{t} \left[\gamma(Y-t\varphi,v,Dv(Y)(I+tD\varphi))-\gamma(Y,v,Dv)\right]= -D_Y\gamma(Y,v,Dv)\cdot\varphi(Y) \nonumber\\
\phantom{aaaaaaaaaaaaaa} + \langle Dv^TD_F\gamma(Y,v,Dv),D\varphi(Y)\rangle. \nonumber \end{eqnarray}

For the second term we note that $\frac{1}{t}(\det(I+tD\varphi)^{-1}-1)$ converges uniformly for $Y\in \O_0$ to $-\tr D\varphi(Y)$. We now prove \eqref{preconvdom}. It suffices to check that
\[|\gamma(Y,p,FC)-\gamma(X,p,F)|\leq N\theta [1+|F|^2+\gamma(X,p,F)] \]
for $|X-Y|<\theta$ and $|C-I|<\theta$. But
\begin{eqnarray}\nonumber
\gamma(Y,p,FC)-\gamma(X,p,F) =\int_0^1 \frac{d}{dt} \gamma\left(Y(t),p,FC(t)\right)dt
\end{eqnarray}
with $Y(t)=(1-t)X+tY$ and $C(t)=(1-t)I+tC$. We note that $|Y(t)-X|<\theta$ and $|C(t)-I|<\theta$. We also have
\begin{eqnarray}
\frac{d}{dt} \gamma(Y(t),p,FC(t))&=&D_X\gamma(Y(t),FC(t))\cdot Y'(t)+\sum_{i,j,k=1}^2 \frac{\p \gamma}{\p F_{ij}}(FC(t)).F_{ik}(C-I)_{kj} \nonumber \\
&=& D_X\gamma(Y(t),p,FC(t))\cdot(Y-X) \nonumber \\
& & \phantom{aaaaaaaaa} +\sum_{j,k}^2[F^T D_F\gamma(Y(t),p,FC(t))]_{kj}(C-I)_{kj} \nonumber
\end{eqnarray}
Thus
\begin{eqnarray}
|\frac{d}{dt} \gamma(Y(t),p,FC(t))| & \leq & |D_X\gamma(Y(t),p,FC(t))||Y-X| \nonumber \\
& &+ \phantom{} |F^TD_F\gamma(Y(t),p,FC(t))||C-I| \nonumber \\
&\leq & N(1+|F|^2+\gamma(X,p,F))\theta. \nonumber
\end{eqnarray}
This concludes the proof.
\end{proof}
We now show that $A(v)= \int_{\O_0} H\left(\frac{v_1}{r}\det D v\right) rdrdz$ satisfies the hypothesis of the previous lemma if $\O$ is an axisymmetric domain such that $\inf_{(x,y,z)\in \O}\sqrt{x^2+y^2}>0$.

\begin{lemma}\label{preconv2}
Let $\O$ be a smooth axisymmetric bounded domain such that \\
 $\inf_{(x,y,z)\in \O}\sqrt{x^2+y^2}>0$. Then  $A(v)= \int_{\O_0} H(\frac{v_1}{r}\det D v) rdrdz$ satisfies the hypotheses of Lemma \ref{preconv1}. Hence for each $\varphi\in C_c^\infty(\O_0,\R^2)$ if we set $v_t(X)=v(X+t\varphi(X))$ then $A(v_t)$ is differentiable at $t=0$ and
\begin{eqnarray}\label{formulaforphi}
\frac{d}{dt}_{|t=0}A(v_t)= \int_{\O_0}-\left[ \frac{-v_1}{r}\det D v H'(\frac{v_1}{r}\det Dv)+H(\frac{v_1}{r}\det D v)\right]\varphi_1drdz + \nonumber \\
\int_{\O_0}\left[H'(\frac{v_1}{r}\det D v)v_1\det D v -H(\frac{v_1}{r}\det D v)r\right]\tr (D\varphi) drdz.
\end{eqnarray}
\end{lemma}

\begin{proof}
Let $\gamma(X,p,F)= H(\frac{p_1}{r}\det F)r$. A direct computation shows that $$F^TD_F\gamma(X,p,F)=(\det F) H'(\frac{p_1}{r}\det F)p_1I.$$ Thus
\begin{eqnarray}
|F^TD_F\gamma(Y,p,FC)|\leq (\det F) (\det C) H'(\frac{p_1}{r'}\det F \det C)p_1 \nonumber
\end{eqnarray}
where we let $Y=(r',z')$. We also let $d:=\det F$ and $c:=\det C$. We want to prove that
\begin{equation}\nonumber
p_1cd|H'(\frac{p_1cd}{r})| \leq M_1(1+d+r H(\frac{p_1}{r}d))
\end{equation}
for some $M_1>0$, for $r,d>0$, for $0\leq p_1\leq M$ (recall that $M$ is a positive real number such that $\|u\|_{L^\infty}\leq M$ if $u$ is in $\mathcal{A}$) and $\frac{1}{4}\leq c\leq 4$. \\

This will prove the result since $\det F \leq |F|^2$ in two dimensions. We use the hypothesis on $H$ to obtain that there exists $M_2>0$ such that for all $r,d>0$ and $ \frac{1}{4}\leq c \leq 4$ we have
\begin{eqnarray}
|H'(\frac{p_1cd}{r'})|  &\leq & M_2(1+ (\frac{p_1cd}{r'})^{-s-1}+(\frac{p_1cd}{r'})^{\tau+1}) \nonumber \\
\end{eqnarray}
Hence using that $0\leq p_1 \leq M$ and that $\frac{1}{4}\leq c\leq 4$ we obtain that, for some $M_2>0$
\begin{equation}\nonumber
p_1cd|H'(\frac{p_1cd}{r'})|  \leq M_2(d+(p_1d)^{-s}{r'}^{s+1}+(p_1)d^{\tau+2}{r'}^{-\tau+1})
\end{equation}
However since $|r-r'|<\theta$ and since there exist $0<r_0<R_0$ such that $r_0<r,r'<R_0$, we can apply the mean value theorem to get that
\begin{eqnarray}
|r^{-\tau-1}-{r'}^{-\tau-1}|&\leq&  M_3|r-r'|\leq M_3 \theta \nonumber \\
|r^{s+1}-{r'}^{s+1}| & \leq & M_3 |r-r'| \leq M_3 \theta, \nonumber
\end{eqnarray}
for some $M_3>0$. We use again that $r_0<r,r'<R_0$ and we can deduce that
\begin{equation}\nonumber
p_1cd|H'(\frac{p_1cd}{r'})|  \leq M_4\left(d+(p_1d)^{-s}{r}^{s+1}+(p_1)d^{\tau+2}{r}^{-\tau+1}\right)
\end{equation}
Now we use the hypothesis on $H$ but this time to obtain a lower bound and we have that, for some $M_5>0$
\begin{eqnarray}
r H(\frac{p_1d}{r}) & \geq & M_5 r\left(1+ (\frac{p_1d}{r})^{-s}+(\frac{p_1d}{r})^{\tau+2} \right) \nonumber\\
& \geq & M_2\left(1 +(p_1d)^{-s}{r}^{s+1}+(p_1)d^{\tau+2}{r}^{-\tau+1} \right).
\end{eqnarray}
This proves that
\[ |F^TD_F\gamma(Y,p,CF)| \leq C[1+|F|^2+\gamma(X,p,F)] .\]
We proceed in the same way to prove that $|D_X\gamma(Y,p,FC)|\leq N(1+|F|^2+\gamma(X,p,F))$ for $|X-Y|<\theta$ and $|C-I|<\theta$ with $\theta$ small enough, the key ingredient being that there exist $0<r_0<R_0$ such that $r_0<r<R_0$. Once we have that $A(v)$ satisfies the hypothesis of Lemma \ref{preconv1} then to compute its derivative at $t=0$ for variations $v_t$ as before we use the formula in that lemma.
\end{proof}
We now deal with the Dirichlet part of the energy.
\begin{lemma}\label{formulaforDir}
Let $F(v):=\int_{\O_0} |D v|^2rdrdz+\int_{\O_0} \frac{v_1^2}{r}drdz$ defined for $v:\O_0 \rightarrow \R^2$. For all $\varphi\in C^1_c(\O_0,\R^2)$ we set $v_t(X)=v(X+t\varphi(X))$. Then $\frac{d}{dt}_{|t=0}F(v_t)$ exists and
\begin{equation}
\frac{d}{dt}_{|t=0}F(v_t)=\int_{\O_0} \langle 2D v^T D v-(|D v|^2+\frac{v_1^2}{r^2})I,D\varphi(X)\rangle rdrdz -\int_{\O_0} \left( |D v|^2-\frac{v_1^2}{r^2} \right)\varphi_1drdz
\end{equation}
\end{lemma}

\begin{proof}
To prove the differentiability of $F(v_t)$ we prove that we can pass to the limit in the quotient $\frac{1}{t}(F(v_t)-F(v))$ by first changing variables to $\tilde{X}=X+t\varphi(X)$ and then applying the Lebesgue dominated convergence. Then we compute the derivative using power series expansion in $t$. The details are left to the reader (compare with the proof of Lemma \ref{preconv1}).
\end{proof}
We conclude the proof of Theorem \ref{Equations}.
\begin{proof}[Proof of Theorem \ref{Equations}]
Thanks to Lemmas \ref{appartenance}, \ref{formulaforDir}, \ref{preconv1} and \ref{preconv2} we can see that for each $\varphi\in C^\infty_c(\O_0,\R^2)$, $G(v_t)$ is differentiable at $t=0$ for $v_t=v(X+t\varphi(X))$ and
\begin{eqnarray}\label{derivativeofG}
\frac{d}{dt}_{|t=0} G(v_t)= \int_{\O_0} \left[\langle 2D v^T D v-(|D v|^2+\frac{v_1^2}{r^2})I,D\varphi(X)\rangle r- \left(|D v|^2-\frac{v_1^2}{r^2} \right)\varphi_1\right]drdz \nonumber \\
+\int_{\O_0}\left[ \frac{v_1}{r}\det D v H'(\frac{v_1}{r}\det D v)-H(\frac{v_1}{r}\det D v)\right]\varphi_1drdz  \nonumber \\
+\int_{\O_0}\left[ \left(H'(\frac{v_1}{r}\det D v)v_1\det D v -H(\frac{v_1}{r}\det D v)r\right)\tr D\varphi\right]rdrdz. \nonumber
\end{eqnarray}
Since $u_{t=0}=u$ minimizes the energy $E$ in $\mathcal{A}_{nc}^{axi}$ we have that $\frac{d}{dt}_{|t=0} G(v_t)=0$. Thus for all $\varphi\in C^\infty_c(\O_0,\R^2)$ we have that the right hand side of \eqref{derivativeofG} must be zero and this is equivalent to equations \eqref{equation1} and \eqref{equation2}. To show that these equations are equivalent to
\[\Dive\left(2Du^TDu+[H'(\det Du)\det Du-|Du|^2-H(\det Du)]I\right)=0 \]
in the sense of distributions we write this last equation in cylindrical coordinates. For that we use the expression of $Du$ in the basis $(e_r,e_\theta,e_z)$ (cf. Appendix)
\[Du=\begin{pmatrix}
\p_r v_1 &0 & \p_z v_1 \\
0 & \frac{v_1}{r} & 0 \\
\p_r v_2 & 0 & \p_zv_2
\end{pmatrix}.\]
Thus
\[Du^TDu =\begin{pmatrix}
|\p_rv|^2 & 0& \p_rv \cdot \p_zv \\
0 & \frac{v_1^2}{r^2} & 0 \\
\p_zv \cdot \p_rv & 0 & |\p_zv|^2
\end{pmatrix}\]
and since $|Du|^2=|Dv|^2+\frac{v_1^2}{r}=|\p_rv|^2+|\p_zv|^2+\frac{v_1^2}{r}$ we find that the energy-momentum tensor
\[T_u:= 2Du^TDu+[H'(\det Du)\det Du-|Du|^2-H(\det Du)]I \]
can be expressed in cylindrical coordinates as
\begin{eqnarray}
T_u= \begin{pmatrix}
|\p_rv|^2-|\p_zv|^2-\frac{v_1^2}{r^2} & 0 & 2\p_rv \cdot \p_z v \\
0 & -|Dv|^2+\frac{v_1^2}{r^2} & 0 \\
2\p_rv \cdot \p_zv & 0 & |\p_zv|^2-|\p_rv|^2-\frac{v_1^2}{r^2}
\end{pmatrix} \nonumber \\
+\left[ H'(\frac{v_1}{r}\det D v)\frac{v_1}{r}\det D v-H(\frac{v_1}{r}\det D v)\right]I.
\end{eqnarray}
By using that if $f$ is vector-valued and if $a,b$ are vector valued then $\Dive(f a\otimes b)=a\otimes b. \nabla \varphi+\varphi D a.b +(\Dive b) a)$ we obtain that for a tensor $A$ we have
\begin{equation}
\Dive A= \begin{pmatrix}
\p_r A_{rr} +\frac{1}{r}\p_\theta A_{r\theta}+\p_zA_{rz}+\frac{A_{rr}-A_{\theta \theta}}{r} \\
\p_r A_{\theta r}+\frac{1}{r}\p_\theta A_{\theta \theta}+\p_z A_{\theta z}+\frac{A_{r \theta}+A_{\theta r}}{r} \\
\p_r A_{zr}+\frac{1}{r}\p_\theta A_{z\theta}+\p_z A_{zz}+\frac{A_{zr}}{r}
\end{pmatrix}.
\end{equation} We thus find that $\dive T_u=0$ corresponds to \eqref{equation1} and \eqref{equation2}.
\end{proof}

\textbf{Acknowledgements:} The authors have been supported by the Millennium Nucleus Center for Analysis of PDE NC130017 of the Chilean Ministry of Economy. D.H. has also been funded by FONDECYT project $\#1150038$ of the Chilean Ministry of Education. We are very grateful for several stimulating discussions with J.M. Ball, P. Bauman, F. Bethuel, J. D\'avila, M. del Pino, R. Jerrard, J. Kristensen, F. Murat, D. Phillips, C. Wang and A. Zarnescu, which greatly helped to shape the research being presented.

\section*{Appendix}

In this appendix we give the expressions of various quantities of interest in cylindrical coordinates. Let $u$ be an axisymmetric map $u(r\cos \theta ,r\sin \theta, z) =v_1(r,z)e_r +v_2(r,z)e_z$, for some $v:\O_0 \rightarrow \R^2$ with $v_1 \geq 0$. We begin with the expression of the Jacobian matrix in cylindrical coordinates (i.e.\ in the basis $(e_r,e_\theta,e_z)$). To derive this expression we use that, by definition, for any $C^1$ curve $\gamma(t)$ we have that
\[\frac{d}{dt} \left[u(\gamma(t))\right]=Du(\gamma(t)).\gamma'(t).\]
We thus consider a $C^1$ curve $\gamma(t)=(r(t),\theta(t),z(t))$ and we compute:
\begin{eqnarray}
\frac{d}{dt} \left[u(\gamma(t))\right] &=& \frac{d}{dt}[v_1(r,z)e_r+v_2(r,z)e_z] \nonumber \\
&=& (\p_rv_1(r,z)\dot{r}+\p_zv_1(r,z)\dot{z})e_r+v_1(r,z)\dot{\theta}e_\theta+ \frac{d}{dt}[v_2(r,z)]e_z \nonumber \\
&= & \left( \p_rv_1 e_r \cdot \dot{\gamma}+\p_zv_1 e_z \cdot \dot{\gamma} \right)e_r +v_1e_\theta\frac{1}{r} e_\theta\cdot \dot{\gamma} + \nonumber \\
& &\phantom{aaaaaaaaaaaaaaaaaaaaaaaaaaaa} \left( \p_rv_2 e_r\cdot \dot{\gamma} +\p_zv_2 e_z\cdot \dot{\gamma} \right) e_z \nonumber
\end{eqnarray}
Now we use that, by definition of the tensorial product of two vectors we have \\
$a \otimes b \cdot h= (b\cdot h)a$. Hence
\begin{equation}
Du=\p_r v_1 e_r \otimes e_r +\p_z v_1 e_r\otimes e_z + \frac{v_1}{r} e_\theta \otimes e_\theta +\p_rv_2 e_z \otimes e_r+ \p_zv_2 e_z \otimes e_z.
\end{equation}
In other words
\begin{equation}
Du= \begin{pmatrix}
\p_rv_1 & 0 & \p_z v_1 \\
0 & \frac{v_1}{r} & 0 \\
\p_rv_2 & 0 & \p_zv_2
\end{pmatrix}.
\end{equation}
We can now deduce that
\begin{equation}
\det Du = \frac{1}{r} v_1 \det D v,
\end{equation}
\begin{equation}
\cof Du = \begin{pmatrix}
\frac{v_1}{r}\p_zv_2 & 0 & -\frac{v_1}{r}\p_rv_2 \\
0 & \det D v & 0 \\
-\frac{v_1}{r}\p_zv_1 & 0 & \frac{v_1}{r}\p_rv_1
\end{pmatrix}
\end{equation}

\begin{equation}
\mathcal{D}(u)=\cof Du^T.u=\frac{v_1}{r}(v\wedge \p_zv, 0 ,-v\wedge \p_rv).
\end{equation}
In conclusion,
\begin{eqnarray}
E(u)&=&2\pi G(v) \nonumber \\
&=& 2\pi \int_{\O_0}\left[\left(|\p_rv|^2+|\p_z v|^2\right)r+\frac{v_1^2}{r}+H(\frac{v_1}{r}\det D v)r \right]drdz. \nonumber
\end{eqnarray}
\bibliographystyle{plain}
\bibliography{Neohookean}
\end{document}